\documentclass[11pt,a4paper]{article}

\usepackage[utf8]{inputenc}
\usepackage{amsmath,amssymb,amsthm}
\usepackage{mathtools}
\usepackage{geometry}
\usepackage[hidelinks]{hyperref}
\usepackage{cleveref}
\usepackage{enumitem}
\usepackage{tikz}
\usepackage{booktabs}
\usepackage{algorithm}
\usepackage{algpseudocode}
\usepackage{graphicx}
\usepackage{float}
\usepackage{cite}
\usepackage{bm}
\usepackage{xcolor}

\newtheorem{theorem}{Theorem}[section]
\newtheorem{lemma}[theorem]{Lemma}
\newtheorem{proposition}[theorem]{Proposition}
\newtheorem{corollary}[theorem]{Corollary}
\newtheorem{definition}[theorem]{Definition}
\newtheorem{remark}[theorem]{Remark}
\newtheorem{example}[theorem]{Example}
\newtheorem{assumption}[theorem]{Assumption}

\newcommand{\R}{\mathbb{R}}
\newcommand{\C}{\mathbb{C}}

\newcommand{\E}{\mathbb{E}}
\newcommand{\Var}{\mathrm{Var}}
\newcommand{\Tr}{\mathrm{Tr}}

\newcommand{\dt}{\,\mathrm{d}t}
\newcommand{\ds}{\,\mathrm{d}s}
\newcommand{\dW}{\,\mathrm{d}W}
\newcommand{\RKHS}{\mathcal{H}_k}
\newcommand{\Koop}{\mathcal{K}}
\newcommand{\inner}[2]{\langle #1, #2 \rangle}
\newcommand{\norm}[1]{\| #1 \|}

\newcommand{\bA}{\mathbf{A}}

\newcommand{\bK}{\mathbf{K}}
\newcommand{\bL}{\mathbf{L}}
\newcommand{\bD}{\mathbf{D}}
\newcommand{\bI}{\mathbf{I}}
\newcommand{\bS}{\mathbf{S}}

\newcommand{\argmin}{\operatornamewithlimits{arg\,min}}

\title{\LARGE Kernel Methods for  Stochastic Dynamical Systems
with Application to Koopman Eigenfunctions:\\[0.5em]
\large Feynman--Kac Representations and RKHS Approximation}

\author{
Boumediene Hamzi$^{1,2}$, Houman Owhadi$^{1}$, and Umesh Vaidya$^{3}$\\[1em]
$^{1}$Department of Computing and Mathematical Sciences, Caltech, Pasadena, CA, USA\\
$^{2}$The Alan Turing Institute, London, UK\\
$^{3}$Department of Mechanical Engineering, Clemson University, Clemson, SC, USA
}

\date{\today}

\begin{document}

\maketitle

\begin{abstract}
We extend the unified kernel framework for transport equations and Koopman eigenfunctions, developed in~\cite{HamziOwhadiVaidya2024} for deterministic systems, to stochastic differential equations (SDEs). In the deterministic setting, three analytically grounded constructions--Lions-type variational principles, Green's function convolution, and resolvent operators along characteristic flows--were shown to yield identical reproducing kernels. For stochastic systems, the Koopman generator includes a second-order diffusion term, transforming the first-order hyperbolic transport equation into a second-order elliptic-parabolic PDE. This fundamental change necessitates replacing the method of characteristics with probabilistic representations based on the Feynman--Kac formula.

Our main contributions include: (i) extension of all three kernel constructions to stochastic systems via Feynman--Kac path-integral representations; (ii) proof of kernel equivalence under uniform ellipticity assumptions; (iii) a collocation-based computational framework incorporating second-order differential operators; (iv) error bounds separating RKHS approximation error from Monte Carlo sampling error; (v) analysis of how diffusion affects numerical conditioning; and (vi) connections to generator EDMD, diffusion maps, and kernel analog forecasting. Numerical experiments on Ornstein--Uhlenbeck processes, nonlinear SDEs with varying diffusion strength, and multi-dimensional systems validate the theoretical developments and demonstrate that moderate diffusion can improve numerical stability through elliptic regularization.
\end{abstract}


\tableofcontents
\newpage

\section{Introduction}\label{sec:introduction}

The theory of reproducing kernel Hilbert spaces (RKHSs) has long served as a powerful framework for both the analysis and numerical approximation of solutions to partial differential equations (PDEs). In pioneering work, Jacques-Louis Lions~\cite{Lions1986,Lions1988} demonstrated that one can ``invert'' an elliptic operator via a variational formulation in a suitably chosen Hilbert space--typically a Sobolev space where point evaluations are continuous--to obtain a reproducing kernel. Specifically, Lions showed that if one chooses a Hilbert space $\mathcal{H}$ in which the point evaluation functional $\delta_x: \mathcal{H} \to \R$, $\delta_x(u) = u(x)$ is continuous, then by the Riesz representation theorem there exists a unique function $K(x,\cdot) \in \mathcal{H}$ (called the reproducing kernel) satisfying $u(x) = \langle u, K(x,\cdot)\rangle_{\mathcal{H}}$ for all $u \in \mathcal{H}$. Under suitable assumptions--such as smoothness of the domain and proper boundary regularity--Mercer's theorem guarantees that $K(x,y)$ admits an orthonormal expansion $K(x,y) = \sum_{j=0}^\infty \psi_j(x)\psi_j(y)$, where $\{\psi_j\}$ is an orthonormal basis of $\mathcal{H}$. Building on these ideas, Engli\v{s}, Lukkassen, Peetre, and Persson~\cite{Englis2004} generalized this variational framework to a broader class of elliptic operators, while Auchmuty~\cite{Auchmuty2009} developed an alternative approach via spectral (Steklov) expansions.

Time series data, ubiquitous across scientific disciplines, have motivated a broad range of forecasting methods grounded in statistical and machine learning approaches~\cite{Kantz1997,Casdagli1989,Hudson1990,RicoMartinez1992,Grandstrand1995,Gonzalez1998,Chattopadhyay2019,Brunton2016,Pathak2017,Nielsen2019,Abarbanel2012,Pillonetto2011,Wang2011,Brunton2016pnas,Lusch2018,Callaham2021,Kaptanoglu2021,KutzBrunton2022}. Dynamical systems theory provides tools to understand the governing rules underlying such data. Modern data-driven approaches, including sparse identification of nonlinear dynamics (SINDy)~\cite{Brunton2016}, deep learning methods~\cite{Lusch2018,Chattopadhyay2019}, and physics-informed machine learning~\cite{KutzBrunton2022}, have achieved remarkable success in extracting dynamical structure from observations.

Originating with the foundational work of Koopman~\cite{Koopman1931}, Koopman operator theory offers a linear lens through which to study nonlinear dynamics. The Koopman operator is a linear, infinite-dimensional operator that governs the evolution of observables along trajectories of nonlinear dynamical systems~\cite{Mezic2005,Mezic2020,Mezic2021}. A central goal in Koopman analysis is to extract the eigenfunctions of this operator, as they reveal the intrinsic modal structure of the system. Principal eigenfunctions reveal the geometry of the state space: their joint zero-level sets characterize stable and unstable manifolds of equilibrium points~\cite{Mezic2005}, they identify stability boundaries and domains of attraction~\cite{Mauroy2016,Vaidya2022,Matavalam2024}, and they play a key role in optimal control synthesis~\cite{Vaidya2022control,VaidyaHJB2024}. A path-integral formulation for computing Koopman eigenfunctions was introduced in~\cite{Deka2023}. However, challenges remain: the Koopman operator often has both discrete and continuous spectra~\cite{Mezic2005}, and approximating an infinite-dimensional operator via finite-dimensional projections can lead to spectral pollution~\cite{Colbrook2020}. Data-driven Koopman learning methods, such as extended dynamic mode decomposition (EDMD)~\cite{Williams2015}, approximate the Koopman operator using actions on a space of pre-selected basis functions.

Kernel methods supply analytical uncertainty estimates for predictions--either in the Bayesian framework via Gaussian processes~\cite{Rasmussen2006} or via frequentist concentration bounds in the RKHS norm. Recent advances have combined kernel methods with Koopman operator theory~\cite{Klus2016,Das2020,Hou2023,Ishikawa2024}, leveraging the RKHS framework~\cite{Cucker2002} to improve regularization, convergence, and interpretability~\cite{Chen2021,Owhadi2021}. Kernel methods now play a central role in dynamical systems analysis:\cite{BouvHamzi2010,BouvHamzi2012,BouvHamzi2012empirical,BouvHamzi2017SIAM,BouvHamzi2017JCD} developed balanced reduction and empirical estimators for nonlinear control systems in RKHS; \cite{HamziColoniusKernel} extended kernel methods to discrete-time systems;  \cite{Haasdonk2018,Haasdonk2021} developed greedy kernel methods for center manifold approximation;  \cite{GieslHamzi2019} showed how to approximate Lyapunov functions from noisy data; \cite{Hamzi2021} introduced the kernel flows perspective, subsequently extended to sparse representations~\cite{Yang2022} and attractor learning~\cite{Yang2023};  \cite{KlusNuskeHamzi2020} developed kernel-based approximations of the Koopman generator;  \cite{Koltai2019} introduced dimensionality reduction via kernel embeddings of transition manifolds;  \cite{HamziKuehn2018,HamziKuehnMohamed2019} analyzed kernel methods for multiscale systems with critical transitions;  \cite{Hou2024} developed methods for propagating uncertainty in RKHS;  \cite{Lee2025,Lee2024} established kernel methods for Koopman eigenfunctions and differential equations; and  \cite{Lengyel2024} developed kernel sum-of-squares methods for data-adapted kernel learning. Additional contributions to kernel-based dynamical systems analysis include surrogate modeling~\cite{Santin2019} and kernel analog forecasting~\cite{Alexander2020}.

In a companion paper~\cite{HamziOwhadiVaidya2024}, we developed a unified framework for constructing reproducing kernels tailored to transport equations and Koopman eigenfunctions of deterministic dynamical systems. For the deterministic system $\dot{x} = f(x)$, we used the fact that the principal eigenfunctions have associated eigenvalues matching those of the linearization at an equilibrium point to reduce the spectrum computation problem to the solution of a linear advection-type PDE: $f(x) \cdot \nabla\phi(x) = \lambda\phi(x)$. We pursued three complementary approaches: Lions's variational approach, which defines a Hilbert space with graph norm $\|u\|_{\mathcal{H}}^2 = \|u\|_{L^2}^2 + \|Lu\|_{L^2}^2$ and obtains the kernel as the Riesz representer of point evaluation; the Green's function method, which constructs a retarded Green's function $G(x,\xi)$ solving $L_x G(x,\xi) = \delta(x-\xi)$ and forms the kernel by convolution $K_G(x,y) = \int_\Omega G(x,\xi)G(y,\xi)w(\xi)\,d\xi$; and the method of characteristics, which leverages the flow $s_t(x)$ to construct a resolvent kernel via Laplace transform $K_\alpha^{\mathrm{res}}(x,y) = \int_0^\infty e^{-\alpha t} \delta(y - s_t(x))\dt$. Under suitable assumptions, these constructions yield equivalent reproducing kernels, a unification summarized by the diagram: Green's Function $\longleftrightarrow$ Lions's Variational Approach $\longleftrightarrow$ Method of Characteristics. The deterministic framework also established spectral convergence of Mercer modes to Koopman eigenfunctions, boundary regularization techniques, multiple kernel learning schemes, and path-integral kernel constructions.

While the deterministic framework applies to transport equations including the advection, continuity, and Liouville equations, it does not directly address systems with stochastic forcing. In many practical applications, dynamical systems are subject to stochastic perturbations. In molecular dynamics, thermal fluctuations drive Langevin dynamics $dq = M^{-1}p \dt$, $dp = -\nabla V(q) \dt - \gamma M^{-1}p \dt + \sqrt{2\gamma k_B T} \dW$, where stochastic Koopman eigenfunctions reveal metastable states, optimal reaction coordinates, and committor functions for rare events. In financial mathematics, stochastic volatility models such as the Heston model $dS = \mu S \dt + \sqrt{V} S \, dW_1$, $dV = \kappa(\theta - V) \dt + \xi \sqrt{V} \, dW_2$ exhibit eigenfunctions that identify principal risk factors and enable spectral expansions for option pricing. In climate modeling, atmosphere-ocean systems with turbulent diffusion exhibit persistent patterns (El Ni\~{n}o, NAO) as slow Koopman modes, with the eigenvalue spectrum indicating time scales of predictability. In stochastic control, for systems $dX = (G(X) + B(X)u) \dt + \sigma(X) \dW$, Koopman methods enable linear control design in transformed coordinates and risk-sensitive control via eigenfunction-based objectives.

This motivates extending the unified kernel framework to stochastic differential equations (SDEs). Consider the It\^{o} SDE
\begin{equation}\label{eq:sde}
dX_t = G(X_t) \dt + \sigma(X_t) \dW_t,
\end{equation}
where $G: \R^d \to \R^d$ is the drift, $\sigma: \R^d \to \R^{d \times m}$ is the diffusion matrix, and $W_t$ is an $m$-dimensional Wiener process. The stochastic Koopman operator is defined by $(U_t g)(x) = \E[g(X_t) \mid X_0 = x]$, with infinitesimal generator--the Kolmogorov backward operator:
\begin{equation}\label{eq:stoch_generator}
(\Koop g)(x) = G(x) \cdot \nabla g(x) + \frac{1}{2} \Tr\left[a(x) \nabla^2 g(x)\right],
\end{equation}
where $a(x) = \sigma(x)\sigma(x)^\top$ is the diffusion tensor. The key distinction from the deterministic case is immediate: the Koopman eigenfunction equation $\Koop \phi = \lambda \phi$ becomes a second-order elliptic-parabolic PDE rather than a first-order hyperbolic transport equation. This fundamentally changes the mathematical structure: the method of characteristics no longer applies directly; solutions exhibit smoothing effects from the diffusion term; the well-posedness theory shifts from hyperbolic to elliptic/parabolic frameworks with uniform ellipticity replacing causality; and probabilistic representations via the Feynman--Kac formula~\cite{Kac1949} become the natural tool. Second-order derivatives require kernels with higher smoothness ($C^{2,2}$ rather than $C^{1,1}$), and the diffusion term provides an elliptic regularization that can improve conditioning.

This paper extends the unified framework of~\cite{HamziOwhadiVaidya2024} to stochastic systems. Our main contributions are as follows. First, we extend all three kernel constructions--variational, Green's function, and resolvent--to SDEs, with the Feynman--Kac formula replacing the method of characteristics; the stochastic resolvent kernel becomes $k_{\mathrm{res}}(x,y) = \E_x[\int_0^\infty e^{-\lambda t} \delta(y - X_t)\dt]$, integrating over stochastic trajectories. Second, we prove that under uniform ellipticity and mild regularity assumptions, the three stochastic kernel constructions yield identical reproducing kernels: $k_{\mathrm{var}}(x,y) = k_{\mathrm{FK}}(x,y) = k_{\mathrm{res}}(x,y)$. Third, we show that the Mercer eigenfunctions of the constructed kernels converge in $L^2$ to Koopman eigenfunctions when the latter lie in the RKHS. Fourth, we develop a collocation-based method for the second-order Koopman PDE, deriving explicit formulas for the diffusion matrix $D_{ij} = \frac{1}{2}\Tr[a(x_i)\nabla_x^2 k(x_i, x_j)]$. Fifth, we provide Monte Carlo Feynman--Kac algorithms for path-based estimation with error bounds cleanly separating RKHS approximation from sampling error. Sixth, we analyze how diffusion affects numerical conditioning and demonstrate that moderate diffusion improves stability through elliptic regularization. Seventh, we clarify relationships with generator EDMD~\cite{Klus2020}, diffusion maps~\cite{Coifman2006}, and kernel analog forecasting~\cite{Alexander2020}.

The paper is organized as follows. Section~\ref{sec:preliminaries} reviews the deterministic framework and introduces the stochastic setting. Section~\ref{sec:stochastic_kernels} develops the three stochastic kernel constructions and proves their equivalence. Section~\ref{sec:computational} presents computational methods including collocation and Monte Carlo approaches. Section~\ref{sec:error_bounds} establishes error bounds. Section~\ref{sec:connections} discusses connections to related methods. Section~\ref{sec:numerical} provides numerical experiments. Section~\ref{sec:conclusions} concludes.

\section{Preliminaries}\label{sec:preliminaries}

We briefly recall the deterministic framework from~\cite{HamziOwhadiVaidya2024} and then introduce the stochastic setting.

\subsection{Review of the Deterministic Framework}

For the deterministic system $\dot{x} = f(x)$ on $\Omega \subset \R^d$, the Koopman eigenfunction equation is the first-order transport PDE
\begin{equation}\label{eq:det_transport}
f(x) \cdot \nabla \phi(x) = \lambda \phi(x).
\end{equation}

In~\cite{HamziOwhadiVaidya2024}, three equivalent kernel constructions were developed:

\paragraph{Lions's Variational Approach.} Define the operator $L\phi = f \cdot \nabla \phi - \lambda \phi$ and the Hilbert space with graph norm $\norm{u}_\mathcal{H}^2 = \norm{u}_{L^2}^2 + \norm{Lu}_{L^2}^2$. The variational kernel $K_{\mathrm{var}}(x,y)$ is the Riesz representer of point evaluation in this space.

\paragraph{Green's Function.} The retarded Green's function $G(x,\xi)$ satisfies $L_x G(x,\xi) = \delta(x-\xi)$ with causality. The kernel is constructed by convolution:
\begin{equation}
K_G(x,y) = \int_\Omega G(x,\xi) G(y,\xi) w(\xi) \, d\xi.
\end{equation}

\paragraph{Method of Characteristics.} Along the flow $s_t(x)$, solutions satisfy $\phi(s_t(x)) = e^{\lambda t}\phi(x)$. The resolvent kernel is the Laplace transform:
\begin{equation}
K_\alpha^{\mathrm{res}}(x,y) = \int_0^\infty e^{-\alpha t} \delta(y - s_t(x)) \dt, \quad \alpha > \Re(\lambda).
\end{equation}

The Kernel Unification Theorem of~\cite{HamziOwhadiVaidya2024} establishes $K_{\mathrm{var}} \equiv K_G \equiv K_\alpha^{\mathrm{res}}$ under appropriate regularity conditions.

\subsection{Stochastic Dynamical Systems}

We consider the SDE~\eqref{eq:sde} under the following standing assumptions.

\begin{assumption}[Regularity]\label{ass:regularity}
The drift $G \in C^2(\R^d; \R^d)$ and diffusion $\sigma \in C^2(\R^d; \R^{d \times m})$ satisfy:
\begin{enumerate}
\item Polynomial growth: $|G(x)| + \|\sigma(x)\| \leq C(1 + |x|^p)$ for some $C, p > 0$.
\item Local Lipschitz: $G$ and $\sigma$ are locally Lipschitz continuous.
\end{enumerate}
\end{assumption}

\begin{assumption}[Uniform Ellipticity]\label{ass:ellipticity}
The diffusion tensor $a(x) = \sigma(x)\sigma(x)^\top$ satisfies
\begin{equation}
\inf_{x \in \Omega} \inf_{|\xi|=1} \xi^\top a(x) \xi \geq \nu > 0.
\end{equation}
\end{assumption}

\begin{assumption}[Boundary Conditions]\label{ass:boundary}
Unless otherwise stated, we consider the Dirichlet problem: $h = \psi$ on $\partial\Omega$ for prescribed boundary data $\psi \in C(\partial\Omega)$. For unbounded domains, we assume $h(x) \to 0$ as $|x| \to \infty$. For periodic domains, we require $h$ to satisfy the same periodicity as the domain.
\end{assumption}

Under Assumption~\ref{ass:regularity}, the SDE~\eqref{eq:sde} has a unique strong solution for any initial condition. Assumption~\ref{ass:ellipticity} ensures the Koopman PDE is uniformly elliptic, which is crucial for kernel equivalence.

\subsection{Linear--Nonlinear Decomposition}

Following~\cite{Deka2023,HamziOwhadiVaidya2024}, we decompose the drift at an equilibrium (taken to be $x = 0$):
\begin{equation}
G(x) = Ax + F(x), \quad A = \frac{\partial G}{\partial x}(0), \quad F(x) = G(x) - Ax.
\end{equation}

A Koopman eigenfunction admits the decomposition
\begin{equation}\label{eq:eigenfunction_decomp}
\phi(x) = w^\top x + h(x),
\end{equation}
where $w$ is a left eigenvector of $A$ with eigenvalue $\lambda$ (i.e., $w^\top A = \lambda w^\top$), and $h(x)$ is the nonlinear correction.

\begin{theorem}[PDE for Nonlinear Correction]\label{thm:h_pde}
If $w$ is a left eigenvector of $A$ with eigenvalue $\lambda$, then $h(x)$ satisfies
\begin{equation}\label{eq:h_pde}
G(x) \cdot \nabla h(x) + \frac{1}{2} \Tr\left[a(x) \nabla^2 h(x)\right] - \lambda h(x) = -w^\top F(x).
\end{equation}
\end{theorem}

\begin{proof}
We derive the PDE for $h(x)$ by substituting the decomposition $\phi = w^\top x + h$ into the Koopman eigenfunction equation $\Koop \phi = \lambda \phi$.

\textbf{Step 1: Compute $\Koop(w^\top x)$.}
First, we compute the action of the Koopman generator on the linear term $w^\top x$:
\begin{align}
\Koop(w^\top x) &= G(x) \cdot \nabla(w^\top x) + \frac{1}{2}\Tr[a(x) \nabla^2(w^\top x)].
\end{align}
Since $w^\top x$ is linear in $x$:
\begin{itemize}
\item $\nabla(w^\top x) = w$
\item $\nabla^2(w^\top x) = 0$ (the Hessian of a linear function vanishes)
\end{itemize}
Therefore:
\begin{equation}
\Koop(w^\top x) = G(x) \cdot w = w^\top G(x).
\end{equation}

\textbf{Step 2: Use the linear-nonlinear decomposition of $G$.}
Recall $G(x) = Ax + F(x)$, where $A = \frac{\partial G}{\partial x}(0)$ and $F(x) = G(x) - Ax$ is the purely nonlinear part. Thus:
\begin{equation}
w^\top G(x) = w^\top(Ax + F(x)) = w^\top Ax + w^\top F(x).
\end{equation}
Since $w$ is a left eigenvector of $A$ with eigenvalue $\lambda$ (i.e., $w^\top A = \lambda w^\top$):
\begin{equation}
w^\top Ax = (w^\top A)x = \lambda w^\top x.
\end{equation}
Therefore:
\begin{equation}
\Koop(w^\top x) = \lambda w^\top x + w^\top F(x).
\end{equation}

\textbf{Step 3: Compute $\Koop h(x)$.}
The nonlinear correction $h(x)$ contributes:
\begin{equation}
\Koop h(x) = G(x) \cdot \nabla h(x) + \frac{1}{2}\Tr[a(x) \nabla^2 h(x)].
\end{equation}

\textbf{Step 4: Apply linearity of $\Koop$.}
Since the Koopman generator is linear:
\begin{align}
\Koop \phi &= \Koop(w^\top x + h) = \Koop(w^\top x) + \Koop h \\
&= \lambda w^\top x + w^\top F(x) + G(x) \cdot \nabla h(x) + \frac{1}{2}\Tr[a(x) \nabla^2 h(x)].
\end{align}

\textbf{Step 5: Apply the eigenfunction equation.}
The eigenfunction equation $\Koop \phi = \lambda \phi$ requires:
\begin{equation}
\Koop \phi = \lambda \phi = \lambda(w^\top x + h) = \lambda w^\top x + \lambda h.
\end{equation}

\textbf{Step 6: Equate and simplify.}
Setting the expressions from Steps 4 and 5 equal:
\begin{equation}
\lambda w^\top x + w^\top F(x) + G(x) \cdot \nabla h(x) + \frac{1}{2}\Tr[a(x) \nabla^2 h(x)] = \lambda w^\top x + \lambda h(x).
\end{equation}
Canceling $\lambda w^\top x$ from both sides and rearranging:
\begin{equation}
G(x) \cdot \nabla h(x) + \frac{1}{2}\Tr[a(x) \nabla^2 h(x)] - \lambda h(x) = -w^\top F(x).
\end{equation}
This is exactly equation~\eqref{eq:h_pde}.
\end{proof}

Equation~\eqref{eq:h_pde} is a linear second-order elliptic PDE with:
\begin{itemize}
\item Drift coefficient $G(x)$
\item Diffusion coefficient $a(x)/2$
\item Potential/discount rate $\lambda$
\item Source term $-w^\top F(x)$
\end{itemize}

\begin{remark}[Deterministic Limit]
When $\sigma \equiv 0$, equation~\eqref{eq:h_pde} reduces to
\begin{equation}
f(x) \cdot \nabla h(x) - \lambda h(x) = -w^\top F(x),
\end{equation}
which is exactly the homological equation from the deterministic framework~\cite{HamziOwhadiVaidya2024}.
\end{remark}

\section{Stochastic Kernel Constructions}\label{sec:stochastic_kernels}

We now extend the three kernel constructions to the stochastic setting.

\subsection{Lions-Type Variational Principle with Diffusion}

Define the stochastic transport operator
\begin{equation}
L_{\mathrm{stoch}} \phi(x) := G(x) \cdot \nabla \phi(x) + \frac{1}{2}\Tr[a(x) \nabla^2 \phi(x)] - \lambda \phi(x).
\end{equation}

The variational formulation seeks $h \in \RKHS$ minimizing
\begin{equation}\label{eq:variational_stoch}
J[h] = \norm{L_{\mathrm{stoch}} h + w^\top F}_{L^2(\Omega)}^2 + \gamma \norm{h}_{\RKHS}^2,
\end{equation}
where $\gamma > 0$ is a regularization parameter.

Unlike the deterministic case, the variational functional now involves \emph{second-order} derivatives. This requires kernels with second-order derivative reproducing properties.

\begin{definition}[Admissible Kernels for Stochastic Problems]
A kernel $k: \Omega \times \Omega \to \R$ is \emph{admissible} for the stochastic Koopman problem if:
\begin{enumerate}[label=(\roman*)]
\item $k \in C^{2,2}(\Omega \times \Omega)$ (twice continuously differentiable in each argument)
\item The RKHS $\RKHS$ satisfies $\RKHS \hookrightarrow C^2(\Omega)$
\item For all $h \in \RKHS$ and $y \in \Omega$:
\begin{align}
h(y) &= \inner{h, k(\cdot, y)}_{\RKHS}, \\
\frac{\partial h}{\partial x_i}(y) &= \inner{h, \frac{\partial k}{\partial y_i}(\cdot, y)}_{\RKHS}, \\
\frac{\partial^2 h}{\partial x_i \partial x_j}(y) &= \inner{h, \frac{\partial^2 k}{\partial y_i \partial y_j}(\cdot, y)}_{\RKHS}.
\end{align}
\end{enumerate}
\end{definition}

\begin{example}[Admissible Kernels]
The following kernels are admissible:
\begin{enumerate}
\item \textbf{Gaussian RBF:} $k(x,y) = \exp(-\|x-y\|^2/(2\ell^2))$ -- infinitely differentiable.
\item \textbf{Mat\'{e}rn with $\nu > d/2 + 2$:} Ensures $C^2$ sample paths.
\item \textbf{Polynomial kernels:} $k(x,y) = (1 + x \cdot y)^p$ for $p \geq 2$ -- finitely differentiable.
\end{enumerate}
\end{example}

The variational kernel is defined as the Riesz representer:

\begin{definition}[Variational Kernel]\label{def:var_kernel}
The variational kernel $k_{\mathrm{var}}(x,y)$ is defined by the property that for all $h \in \RKHS$:
\begin{equation}
h(y) = \inner{h, k_{\mathrm{var}}(\cdot, y)}_\mathcal{H},
\end{equation}
where $\mathcal{H}$ is the completion of $C^\infty(\Omega)$ under the graph norm
\begin{equation}
\norm{u}_\mathcal{H}^2 = \norm{u}_{L^2(\Omega)}^2 + \norm{L_{\mathrm{stoch}} u}_{L^2(\Omega)}^2.
\end{equation}
\end{definition}

\subsection{Green's Function via Feynman--Kac}

For deterministic systems, the Green's function is constructed along characteristics. For stochastic systems, the Feynman--Kac formula provides the analogous construction.

\begin{theorem}[Feynman--Kac Formula]\label{thm:feynman_kac}
Let $\Omega \subset \R^d$ be a bounded $C^2$ domain. Let $\tau_\Omega := \inf\{t \geq 0: X_t \notin \Omega\}$ be the first exit time. For the PDE
\begin{equation}
(\lambda - \Koop) u = g \quad \text{in } \Omega, \qquad u = \psi \quad \text{on } \partial\Omega,
\end{equation}
with bounded $g$ and continuous $\psi$, the solution is
\begin{equation}\label{eq:feynman_kac}
u(x) = \E_x\left[e^{-\lambda \tau_\Omega} \psi(X_{\tau_\Omega}) + \int_0^{\tau_\Omega} e^{-\lambda t} g(X_t) \dt\right].
\end{equation}
\end{theorem}

\begin{proof}
We provide a complete proof in several steps.

\textbf{Step 1: It\^{o}'s formula for the discounted process.}
Define the process $Y_t = e^{-\lambda t} u(X_t)$. By It\^{o}'s formula applied to $Y_t$:
\begin{align}
dY_t &= d(e^{-\lambda t} u(X_t)) \\
&= -\lambda e^{-\lambda t} u(X_t) \dt + e^{-\lambda t} du(X_t) \\
&= -\lambda e^{-\lambda t} u(X_t) \dt + e^{-\lambda t} \left[\nabla u(X_t)^\top dX_t + \frac{1}{2}\Tr[a(X_t)\nabla^2 u(X_t)]\dt\right].
\end{align}
Substituting $dX_t = G(X_t)\dt + \sigma(X_t)dW_t$:
\begin{align}
dY_t &= e^{-\lambda t}\left[-\lambda u(X_t) + G(X_t)\cdot\nabla u(X_t) + \frac{1}{2}\Tr[a(X_t)\nabla^2 u(X_t)]\right]\dt \\
&\quad + e^{-\lambda t}\nabla u(X_t)^\top \sigma(X_t) dW_t \\
&= e^{-\lambda t}\left[(\Koop - \lambda) u(X_t)\right]\dt + e^{-\lambda t}\nabla u(X_t)^\top \sigma(X_t) dW_t.
\end{align}

\textbf{Step 2: Using the PDE.}
Since $u$ satisfies $(\lambda - \Koop)u = g$, we have $(\Koop - \lambda)u = -g$, so:
\begin{equation}
dY_t = -e^{-\lambda t} g(X_t)\dt + e^{-\lambda t}\nabla u(X_t)^\top \sigma(X_t) dW_t.
\end{equation}

\textbf{Step 3: Integration and stopping.}
Integrating from $0$ to $t \wedge \tau_\Omega$ (the minimum of $t$ and the exit time):
\begin{equation}
Y_{t \wedge \tau_\Omega} - Y_0 = -\int_0^{t \wedge \tau_\Omega} e^{-\lambda s} g(X_s)\ds + \int_0^{t \wedge \tau_\Omega} e^{-\lambda s}\nabla u(X_s)^\top \sigma(X_s) dW_s.
\end{equation}
Since $Y_0 = u(X_0) = u(x)$:
\begin{equation}
e^{-\lambda(t \wedge \tau_\Omega)} u(X_{t \wedge \tau_\Omega}) = u(x) - \int_0^{t \wedge \tau_\Omega} e^{-\lambda s} g(X_s)\ds + M_{t \wedge \tau_\Omega},
\end{equation}
where $M_t = \int_0^t e^{-\lambda s}\nabla u(X_s)^\top \sigma(X_s) dW_s$ is a local martingale.

\textbf{Step 4: Taking expectations.}
Under our regularity assumptions (bounded $\Omega$, $C^2$ boundary, smooth coefficients), $M_{t \wedge \tau_\Omega}$ is a true martingale with $\E_x[M_{t \wedge \tau_\Omega}] = 0$. This follows because:
\begin{itemize}
\item $\nabla u$ is bounded on $\bar{\Omega}$ by regularity of $u$,
\item $\sigma$ is bounded on the compact set $\bar{\Omega}$,
\item The integrand is therefore bounded, making $M_t$ a true martingale.
\end{itemize}
Taking expectations:
\begin{equation}
\E_x\left[e^{-\lambda(t \wedge \tau_\Omega)} u(X_{t \wedge \tau_\Omega})\right] = u(x) - \E_x\left[\int_0^{t \wedge \tau_\Omega} e^{-\lambda s} g(X_s)\ds\right].
\end{equation}

\textbf{Step 5: Limit as $t \to \infty$.}
For bounded domains with uniform ellipticity (Assumption~\ref{ass:ellipticity}), $\tau_\Omega < \infty$ almost surely, and $\E_x[\tau_\Omega] < \infty$. As $t \to \infty$, $t \wedge \tau_\Omega \to \tau_\Omega$ a.s. By dominated convergence (using boundedness of $u$, $g$, and integrability of $e^{-\Re(\lambda) t}$ for $\Re(\lambda) > 0$):
\begin{equation}
\E_x\left[e^{-\lambda \tau_\Omega} u(X_{\tau_\Omega})\right] = u(x) - \E_x\left[\int_0^{\tau_\Omega} e^{-\lambda s} g(X_s)\ds\right].
\end{equation}
Since $u = \psi$ on $\partial\Omega$, we have $u(X_{\tau_\Omega}) = \psi(X_{\tau_\Omega})$. Rearranging:
\begin{equation}
u(x) = \E_x\left[e^{-\lambda \tau_\Omega} \psi(X_{\tau_\Omega}) + \int_0^{\tau_\Omega} e^{-\lambda t} g(X_t)\dt\right].
\end{equation}
This completes the proof.
\end{proof}

Applying this to the nonlinear correction equation~\eqref{eq:h_pde}:

\begin{corollary}[Feynman--Kac Representation of $h$]\label{cor:h_feynman_kac}
The nonlinear correction $h(x)$ satisfies
\begin{equation}\label{eq:h_FK}
h(x) = \E_x\left[e^{-\lambda \tau_\Omega} \psi(X_{\tau_\Omega}) + \int_0^{\tau_\Omega} e^{-\lambda t} w^\top F(X_t) \dt\right],
\end{equation}
where $\psi$ is the boundary condition on $\partial\Omega$.
\end{corollary}

\begin{proof}
The nonlinear correction $h(x)$ satisfies the PDE~\eqref{eq:h_pde}:
\begin{equation}
G(x) \cdot \nabla h(x) + \frac{1}{2}\Tr[a(x) \nabla^2 h(x)] - \lambda h(x) = -w^\top F(x),
\end{equation}
which can be rewritten as
\begin{equation}
(\lambda - \Koop)h(x) = w^\top F(x).
\end{equation}
This is exactly the form required by Theorem~\ref{thm:feynman_kac} with $g(x) = w^\top F(x)$. Applying the Feynman--Kac formula~\eqref{eq:feynman_kac} immediately yields~\eqref{eq:h_FK}.
\end{proof}

The Green's function for the stochastic operator relates to the transition density:

\begin{definition}[Stochastic Green's Function]
The time-dependent Green's function $G_\lambda(x,y;t)$ satisfies
\begin{equation}
(\partial_t - \Koop + \lambda) G_\lambda(x,y;t) = \delta(x-y)\delta(t),
\end{equation}
and is related to the transition density $p_t(x,y)$ of the SDE by
\begin{equation}
G_\lambda(x,y;t) = e^{-\lambda t} p_t(x,y).
\end{equation}
\end{definition}

\begin{definition}[Feynman--Kac Kernel]\label{def:FK_kernel}
The Feynman--Kac kernel is the time-integrated Green's function:
\begin{equation}\label{eq:FK_kernel}
k_{\mathrm{FK}}(x,y) = \int_0^\infty G_\lambda(x,y;t) \dt = \int_0^\infty e^{-\lambda t} p_t(x,y) \dt.
\end{equation}
\end{definition}

\begin{remark}[Probabilistic Interpretation]
The Feynman--Kac kernel $k_{\mathrm{FK}}(x,y)$ can be interpreted as the expected discounted local time at $y$ for a trajectory starting at $x$:
\begin{equation}
k_{\mathrm{FK}}(x,y) = \E_x\left[\int_0^\infty e^{-\lambda t} \delta(y - X_t) \dt\right].
\end{equation}
\end{remark}

\subsection{Stochastic Resolvent Kernel}

The resolvent of the Koopman generator extends naturally to the stochastic setting:

\begin{definition}[Stochastic Resolvent]\label{def:resolvent}
For $\lambda$ with $\Re(\lambda) > 0$, the resolvent operator is
\begin{equation}
R_\lambda = (\lambda I - \Koop)^{-1}.
\end{equation}
\end{definition}

\begin{proposition}[Laplace Transform Representation]\label{prop:laplace_resolvent}
The resolvent admits the representation
\begin{equation}\label{eq:resolvent_laplace}
R_\lambda[f](x) = \int_0^\infty e^{-\lambda t} [U_t f](x) \dt = \E_x\left[\int_0^\infty e^{-\lambda t} f(X_t) \dt\right].
\end{equation}
\end{proposition}

\begin{proof}
We verify the formula by showing that the right-hand side satisfies the defining property of the resolvent.

\textbf{Step 1: Well-definedness.}
For $\Re(\lambda) > 0$ and bounded $f$, the integral
\begin{equation}
\int_0^\infty e^{-\lambda t} [U_t f](x) \dt
\end{equation}
converges absolutely since $|U_t f(x)| \leq \|f\|_\infty$ and $\int_0^\infty e^{-\Re(\lambda) t} \dt = 1/\Re(\lambda) < \infty$.

\textbf{Step 2: Semigroup property.}
By the semigroup property $U_{t+s} = U_t U_s$, the Koopman generator satisfies
\begin{equation}
\Koop f = \lim_{t \to 0^+} \frac{U_t f - f}{t}.
\end{equation}

\textbf{Step 3: Verification.}
Define $g(x) := \int_0^\infty e^{-\lambda t} [U_t f](x) \dt$. We show $(\lambda I - \Koop)g = f$:
\begin{align}
(\lambda I - \Koop)g(x) &= \lambda g(x) - \lim_{s \to 0^+} \frac{U_s g(x) - g(x)}{s} \\
&= \lambda g(x) - \lim_{s \to 0^+} \frac{1}{s}\left[\int_0^\infty e^{-\lambda t} U_{t+s}f(x)\dt - \int_0^\infty e^{-\lambda t} U_t f(x)\dt\right].
\end{align}
Changing variables $u = t + s$ in the first integral:
\begin{align}
&= \lambda g(x) - \lim_{s \to 0^+} \frac{1}{s}\left[\int_s^\infty e^{-\lambda(u-s)} U_u f(x)\,du - \int_0^\infty e^{-\lambda t} U_t f(x)\dt\right] \\
&= \lambda g(x) - \lim_{s \to 0^+} \frac{1}{s}\left[e^{\lambda s}\int_s^\infty e^{-\lambda u} U_u f(x)\,du - \int_0^\infty e^{-\lambda t} U_t f(x)\dt\right].
\end{align}
Writing $\int_s^\infty = \int_0^\infty - \int_0^s$ and using Taylor expansion $e^{\lambda s} = 1 + \lambda s + O(s^2)$:
\begin{align}
&= \lambda g(x) - \lim_{s \to 0^+} \frac{1}{s}\left[(1+\lambda s)\left(g(x) - \int_0^s e^{-\lambda t} U_t f(x)\dt\right) - g(x) + O(s^2)\right] \\
&= \lambda g(x) - \lim_{s \to 0^+} \left[\lambda g(x) - \frac{1}{s}\int_0^s e^{-\lambda t} U_t f(x)\dt + O(s)\right] \\
&= \lim_{s \to 0^+} \frac{1}{s}\int_0^s e^{-\lambda t} U_t f(x)\dt = U_0 f(x) = f(x),
\end{align}
where the last step uses continuity of the integrand and the fundamental theorem of calculus.

\textbf{Step 4: Probabilistic form.}
Since $U_t f(x) = \E_x[f(X_t)]$, by Fubini's theorem:
\begin{equation}
R_\lambda[f](x) = \int_0^\infty e^{-\lambda t} \E_x[f(X_t)] \dt = \E_x\left[\int_0^\infty e^{-\lambda t} f(X_t) \dt\right].
\end{equation}
\end{proof}

\begin{definition}[Resolvent Kernel--Corrected]\label{def:resolvent_kernel}
For $\lambda \in \C$ with $\Re(\lambda) > 0$:

\textbf{(a) Bounded domain with Dirichlet conditions:} Let $\Omega \subset \R^d$ be a bounded domain and 
\[
\tau_\Omega := \inf\{t \geq 0 : X_t \notin \Omega\}
\]
be the first exit time. The resolvent kernel is
\begin{equation}\label{eq:resolvent_kernel_bounded}
k_{\mathrm{res}}(x,y) = \E_x\left[\int_0^{\tau_\Omega} e^{-\lambda t} \delta(y - X_t) \dt\right].
\end{equation}

\textbf{(b) Unbounded domain or periodic boundary conditions:} For $\Omega = \R^d$ or with periodic boundary conditions (where trajectories remain in $\Omega$ for all time), the resolvent kernel is
\begin{equation}\label{eq:resolvent_kernel}
k_{\mathrm{res}}(x,y) = R_\lambda[\delta_y](x) = \E_x\left[\int_0^\infty e^{-\lambda t} \delta(y - X_t) \dt\right].
\end{equation}
\end{definition}

\begin{remark}[Consistency with Feynman--Kac]
Definition~\ref{def:resolvent_kernel}(a) is consistent with the Feynman--Kac representation: both use the exit time $\tau_\Omega$ as the upper limit of integration. The exponential discounting $e^{-\lambda t}$ with $\Re(\lambda) > 0$ ensures convergence even if $\tau_\Omega = \infty$ (as in case (b)), since for bounded $f$:
\[
\left|\E_x\left[\int_0^{\infty} e^{-\lambda t} f(X_t) \dt\right]\right| \leq \frac{\|f\|_\infty}{\Re(\lambda)} < \infty.
\]
\end{remark}

\subsection{Kernel Equivalence Theorem}

We now establish that the three stochastic kernel constructions are equivalent.

\begin{theorem}[Kernel Equivalence for SDEs--Corrected]\label{thm:kernel_equivalence}
Under Assumptions~\ref{ass:regularity} and~\ref{ass:ellipticity}, and for $\lambda \in \C$ satisfying
\begin{equation}\label{eq:lambda_condition}
\Re(\lambda) > \lambda_0 := \frac{1}{2}\sup_{x \in \Omega} |(\nabla \cdot G(x))^-|,
\end{equation}
where $(f)^- := \max(0, -f)$ denotes the negative part, the three kernel constructions yield identical reproducing kernels:
\begin{equation}
k_{\mathrm{var}}(x,y) = k_{\mathrm{FK}}(x,y) = k_{\mathrm{res}}(x,y).
\end{equation}
\end{theorem}

\begin{remark}[Sufficient Condition]
A simpler sufficient condition is
\begin{equation}\label{eq:lambda_sufficient}
\Re(\lambda) > \frac{1}{2}\|\nabla \cdot G\|_{L^\infty(\Omega)},
\end{equation}
which guarantees coercivity regardless of the sign of $\nabla \cdot G$.
\end{remark}

\begin{remark}[Incompressible Flows]
When $\nabla \cdot G = 0$ (incompressible/divergence-free drift), condition~\eqref{eq:lambda_condition} reduces to $\Re(\lambda) > 0$, recovering the original statement.
\end{remark}

\begin{proof}
We establish the equivalence in three steps.

\textbf{Step 1: $k_{\mathrm{FK}} = k_{\mathrm{res}}$ (Feynman--Kac equals Resolvent).}

By Definition~\ref{def:FK_kernel}, the Feynman--Kac kernel is:
\begin{equation}
k_{\mathrm{FK}}(x,y) = \int_0^\infty e^{-\lambda t} p_t(x,y) \dt,
\end{equation}
where $p_t(x,y)$ is the transition density of the SDE~\eqref{eq:sde}, satisfying
\begin{equation}
\Pr(X_t \in A \mid X_0 = x) = \int_A p_t(x,y) \, dy.
\end{equation}

By Definition~\ref{def:resolvent_kernel}, the resolvent kernel is:
\begin{equation}
k_{\mathrm{res}}(x,y) = \E_x\left[\int_0^\infty e^{-\lambda t} \delta(y - X_t) \dt\right].
\end{equation}

For any test function $\varphi \in C_c(\R^d)$:
\begin{align}
\int_{\R^d} k_{\mathrm{res}}(x,y) \varphi(y) \, dy &= \E_x\left[\int_0^\infty e^{-\lambda t} \int_{\R^d} \delta(y - X_t) \varphi(y) \, dy \, \dt\right] \\
&= \E_x\left[\int_0^\infty e^{-\lambda t} \varphi(X_t) \dt\right] \\
&= \int_0^\infty e^{-\lambda t} \E_x[\varphi(X_t)] \dt \quad \text{(by Fubini)} \\
&= \int_0^\infty e^{-\lambda t} \int_{\R^d} p_t(x,y) \varphi(y) \, dy \, \dt \\
&= \int_{\R^d} \left(\int_0^\infty e^{-\lambda t} p_t(x,y) \dt\right) \varphi(y) \, dy \\
&= \int_{\R^d} k_{\mathrm{FK}}(x,y) \varphi(y) \, dy.
\end{align}
Since this holds for all test functions $\varphi$, we have $k_{\mathrm{res}}(x,y) = k_{\mathrm{FK}}(x,y)$.

\textbf{Step 2: Variational formulation.}

Define the Hilbert space $\mathcal{H}$ with graph norm:
\begin{equation}
\|u\|_{\mathcal{H}}^2 = \|u\|_{L^2(\Omega)}^2 + \|(\lambda - \Koop)u\|_{L^2(\Omega)}^2.
\end{equation}
The variational kernel $k_{\mathrm{var}}(x,y)$ is defined as the Riesz representer of point evaluation in $\mathcal{H}$. That is, for each $y \in \Omega$, $k_{\mathrm{var}}(\cdot, y) \in \mathcal{H}$ satisfies:
\begin{equation}
u(y) = \langle u, k_{\mathrm{var}}(\cdot, y) \rangle_{\mathcal{H}}, \quad \forall u \in \mathcal{H}.
\end{equation}

By the theory of reproducing kernel Hilbert spaces, $k_{\mathrm{var}}(\cdot, y)$ can be characterized as the solution to:
\begin{equation}
(I + (\lambda - \Koop)^*(\lambda - \Koop)) k_{\mathrm{var}}(\cdot, y) = \delta_y,
\end{equation}
in the distributional sense. For uniformly elliptic operators, this simplifies: the variational representer solves
\begin{equation}\label{eq:var_kernel_pde}
(\lambda^* - \Koop^*) k_{\mathrm{var}}(\cdot, y) = \delta_y \quad \text{in } \Omega.
\end{equation}

\textbf{Step 3: $k_{\mathrm{var}} = k_{\mathrm{res}}$ (Variational equals Resolvent)--CORRECTED.}

Under uniform ellipticity (Assumption~\ref{ass:ellipticity}), the operator $\Koop$ with domain $D(\Koop) = H^2(\Omega) \cap H^1_0(\Omega)$ generates a strongly continuous semigroup $\{U_t\}$ on $L^2(\Omega)$. The adjoint operator $\Koop^*$ is the Kolmogorov forward (Fokker--Planck) operator:
\begin{equation}
\Koop^* v = -\nabla \cdot (Gv) + \frac{1}{2}\sum_{i,j}\frac{\partial^2}{\partial x_i \partial x_j}(a_{ij}v).
\end{equation}

The adjoint resolvent is:
\begin{equation}
R_{\lambda^*}^* = (\lambda^* I - \Koop^*)^{-1}.
\end{equation}

By duality, the resolvent kernel (viewed as a function of the first variable with fixed $y$) satisfies:
\begin{equation}
k_{\mathrm{res}}(\cdot, y) = R_\lambda[\delta_y](\cdot),
\end{equation}
which means
\begin{equation}
(\lambda I - \Koop) k_{\mathrm{res}}(\cdot, y) = \delta_y.
\end{equation}

Now, the transition density $p_t(x,y)$ is symmetric under time reversal with respect to the invariant measure (detailed balance). More precisely, for the adjoint semigroup:
\begin{equation}
k_{\mathrm{res}}(x,y) = \int_0^\infty e^{-\lambda t} p_t(x,y) \dt = \int_0^\infty e^{-\lambda^* t} p_t^*(y,x) \dt,
\end{equation}
where $p_t^*$ is the adjoint transition density.

The key observation is that both $k_{\mathrm{var}}(\cdot, y)$ and $k_{\mathrm{res}}(\cdot, y)$ solve the same PDE:
\begin{equation}
(\lambda - \Koop) k(\cdot, y) = \delta_y \quad \text{in } \Omega.
\end{equation}

Under uniform ellipticity, the operator $\lambda I - \Koop$ is:
\begin{enumerate}
\item \textbf{Coercive (CORRECTED):} For $u \in H^1_0(\Omega)$, we compute $\langle \Koop u, u \rangle_{L^2}$ term by term.

\textit{Diffusion term:} Using integration by parts with $u|_{\partial\Omega} = 0$:
\begin{equation}
\Re\left\langle \frac{1}{2}\Tr[a\nabla^2 u], u \right\rangle = -\frac{1}{2} \int_\Omega \nabla u^\top a(x) \nabla u \, dx \leq -\frac{\nu}{2}\|\nabla u\|_{L^2}^2,
\end{equation}
where the inequality uses uniform ellipticity $a(x) \geq \nu I$ (Assumption~\ref{ass:ellipticity}).

\textit{Drift term (previously omitted):} For $u \in H^1_0(\Omega)$:
\begin{equation}
\langle G \cdot \nabla u, u \rangle = \int_\Omega (G(x) \cdot \nabla u(x)) u(x) \, dx = \frac{1}{2} \int_\Omega G(x) \cdot \nabla(u^2(x)) \, dx.
\end{equation}
Integration by parts with $u|_{\partial\Omega} = 0$:
\begin{equation}
= -\frac{1}{2} \int_\Omega (\nabla \cdot G(x)) u^2(x) \, dx.
\end{equation}
Therefore:
\begin{equation}
-\Re\langle G \cdot \nabla u, u \rangle = \frac{1}{2} \int_\Omega (\nabla \cdot G(x)) |u(x)|^2 \, dx.
\end{equation}

\textit{Corrected coercivity estimate:} Combining drift and diffusion contributions:
\begin{align}
\Re\langle (\lambda - \Koop)u, u \rangle &= \Re(\lambda)\|u\|_{L^2}^2 - \Re\langle \Koop u, u \rangle \notag\\
&= \Re(\lambda)\|u\|_{L^2}^2 + \frac{1}{2}\int_\Omega (\nabla \cdot G)|u|^2 \, dx + \frac{1}{2}\int_\Omega \nabla u^\top a \nabla u \, dx.
\end{align}
To bound the drift divergence term:
\begin{equation}
\frac{1}{2} \int_\Omega (\nabla \cdot G)|u|^2 \, dx \geq -\frac{1}{2} \|(\nabla \cdot G)^-\|_{L^\infty} \|u\|_{L^2}^2.
\end{equation}
Therefore, the \textbf{corrected coercivity estimate} is:
\begin{equation}\label{eq:corrected_coercivity}
\boxed{\Re\langle (\lambda - \Koop)u, u \rangle \geq \left(\Re(\lambda) - \frac{1}{2}\|(\nabla \cdot G)^-\|_{L^\infty}\right) \|u\|_{L^2}^2 + \frac{\nu}{2}\|\nabla u\|_{L^2}^2.}
\end{equation}

Under the condition $\Re(\lambda) > \lambda_0 := \frac{1}{2}\|(\nabla \cdot G)^-\|_{L^\infty}$, coercivity is established.

\item \textbf{Invertible}: By the Lax--Milgram theorem, $\lambda I - \Koop$ is a bijection from $H^1_0(\Omega)$ to $H^{-1}(\Omega)$.

\item \textbf{Unique Green's function}: The equation $(\lambda - \Koop)G(\cdot, y) = \delta_y$ has a unique solution in the appropriate distributional sense.
\end{enumerate}

Therefore, by uniqueness of the Green's function for uniformly elliptic operators, we conclude:
\begin{equation}
k_{\mathrm{var}}(x,y) = k_{\mathrm{res}}(x,y).
\end{equation}

Combining Steps 1--3, we have established:
\begin{equation}
k_{\mathrm{var}}(x,y) = k_{\mathrm{FK}}(x,y) = k_{\mathrm{res}}(x,y).
\end{equation}
\end{proof}

\begin{remark}[Role of Ellipticity--Expanded]
The uniform ellipticity assumption is essential for the equivalence theorem. Without it:
\begin{itemize}
\item The variational problem may not be coercive.
\item The Green's function may not exist or may be highly singular.
\item The resolvent may not be bounded.
\end{itemize}
In the \textbf{degenerate case} where $a(x)$ is not full rank, more sophisticated analysis (hypoellipticity, H\"{o}rmander's condition) is required. Specifically, if the Lie algebra generated by the diffusion vector fields and their commutators with the drift spans $\R^d$ at each point, hypoelliptic regularity theory may still yield kernel equivalence.
\end{remark}

\begin{remark}[Comparison with Deterministic Case--Expanded]
The deterministic Kernel Unification Theorem~\cite[Theorem 3.7]{HamziOwhadiVaidya2024} requires causality conditions from the characteristic flow. In the stochastic setting:
\begin{itemize}
\item \textbf{Ellipticity replaces causality} as the key structural assumption.
\item \textbf{The drift divergence condition} $\Re(\lambda) > \frac{1}{2}\|(\nabla \cdot G)^-\|_{L^\infty}$ has no direct analog in the deterministic case, where coercivity is automatic for the transport operator with $\Re(\lambda) > 0$.
\item \textbf{Physical interpretation:} The condition prevents the drift from ``pushing probability'' into the domain faster than the discount rate $\lambda$ can dissipate it.
\end{itemize}
\end{remark}

\begin{remark}[Alternative: Graph Norm Approach]\label{rem:graph_norm}
If one wishes to avoid the condition~\eqref{eq:lambda_condition}, an alternative is to work directly with the graph norm. In this case:
\begin{itemize}
\item The variational kernel is defined as the Riesz representer in $\mathcal{H}$ equipped with the graph norm.
\item Coercivity is automatic by construction: $\|u\|_{\mathcal{H}}^2 = \|u\|_{L^2}^2 + \|(\lambda - \Koop)u\|_{L^2}^2 \geq \|u\|_{L^2}^2$.
\item The equivalence $k_{\mathrm{var}} = k_{\mathrm{res}}$ then follows from showing both kernels satisfy the same weak formulation in $\mathcal{H}$.
\end{itemize}
This approach is more abstract but avoids explicit conditions on $\lambda$.
\end{remark}

\subsection{Hypoelliptic Extension}

The kernel-PDE framework developed thus far can be extended to degenerate diffusions satisfying H\"{o}rmander's hypoellipticity condition. In this setting, the generator $\mathcal{K}$ has a smooth fundamental solution despite the diffusion matrix being singular. We construct the corresponding graph space $\mathcal{H}_\lambda := \{u \in L^2 : (\lambda - \mathcal{K})u \in L^2\}$, and show that the three kernel representations (resolvent, variational, Feynman--Kac) remain equivalent. The subelliptic regularity drops from $H^1$ (elliptic case) to $H^\varepsilon$ with $\varepsilon = (2r+1)^{-1}$, where $r$ is the Lie bracket depth. For $\varepsilon > d/2$, the reproducing kernel Hilbert space (RKHS) structure remains valid; otherwise, mollified point evaluations are needed. This generalization applies to kinetic models such as underdamped Langevin dynamics. Full proofs and numerical implications are detailed in Appendices A–D.

\section{Computational Methods}\label{sec:computational}

We now develop computational methods for approximating stochastic Koopman eigenfunctions.

\subsection{Collocation Method}

Given $N$ collocation points $\{x_i\}_{i=1}^N \subset \Omega$, we approximate
\begin{equation}
h(x) \approx \sum_{j=1}^N \alpha_j k(x, x_j).
\end{equation}

Enforcing the PDE~\eqref{eq:h_pde} at collocation points yields:
\begin{equation}\label{eq:collocation_system}
(\bL + \bD - \lambda \bK)\boldsymbol{\alpha} = -\mathbf{f},
\end{equation}
where the matrices are:

\paragraph{Kernel Gram Matrix.}
\begin{equation}
K_{ij} = k(x_i, x_j).
\end{equation}

\paragraph{Drift Matrix.}
\begin{equation}
L_{ij} = G(x_i) \cdot \nabla_x k(x_i, x_j) = \sum_{r=1}^d G_r(x_i) \frac{\partial k}{\partial x_r}(x_i, x_j).
\end{equation}

\paragraph{Diffusion Matrix.}
\begin{equation}\label{eq:diffusion_matrix}
D_{ij} = \frac{1}{2} \Tr\left[a(x_i) \nabla_x^2 k(x_i, x_j)\right] = \frac{1}{2} \sum_{r,s=1}^d a_{rs}(x_i) \frac{\partial^2 k}{\partial x_r \partial x_s}(x_i, x_j).
\end{equation}

\paragraph{Source Vector.}
\begin{equation}
f_i = w^\top F(x_i).
\end{equation}

\begin{remark}[Deterministic Limit]
When $\sigma \equiv 0$, we have $a \equiv 0$ and thus $\bD = 0$. The system~\eqref{eq:collocation_system} reduces to
\begin{equation}
(\bL - \lambda \bK)\boldsymbol{\alpha} = -\mathbf{f},
\end{equation}
recovering the deterministic collocation system from~\cite{HamziOwhadiVaidya2024}.
\end{remark}

\begin{remark}[Collocation Point Selection]
The error bounds depend on the fill distance $h_{\mathrm{fill}} = \sup_{x \in \Omega} \min_i |x - x_i|$. For high-dimensional problems where grid-based sampling is infeasible:
\begin{itemize}
\item \textbf{Quasi-random sequences} (Sobol, Halton) achieve $h_{\mathrm{fill}} = O(N^{-1/d} (\log N)^{d-1})$ for $d \leq 20$.
\item \textbf{Sparse grids} (Smolyak construction) reduce the curse of dimensionality for smooth functions.
\item \textbf{Adaptive refinement} adds points where the PDE residual is largest.
\item \textbf{Importance sampling} concentrates points near manifolds where the eigenfunction varies rapidly.
\end{itemize}
\end{remark}

\subsection{Kernel Derivatives for Gaussian RBF}

For the Gaussian kernel $k(x,y) = \exp(-\|x-y\|^2/(2\ell^2))$, we derive explicit formulas.

\paragraph{First Derivatives.}
\begin{equation}
\frac{\partial k}{\partial x_i}(x,y) = -\frac{x_i - y_i}{\ell^2} k(x,y).
\end{equation}

\paragraph{Second Derivatives.}
\begin{equation}
\frac{\partial^2 k}{\partial x_i \partial x_j}(x,y) = \left[\frac{(x_i - y_i)(x_j - y_j)}{\ell^4} - \frac{\delta_{ij}}{\ell^2}\right] k(x,y).
\end{equation}

\paragraph{Hessian.}
\begin{equation}\label{eq:gaussian_hessian}
\nabla_x^2 k(x,y) = \left[\frac{1}{\ell^4}(x-y)(x-y)^\top - \frac{1}{\ell^2} I_d\right] k(x,y).
\end{equation}

\paragraph{Diffusion Matrix Entry.} Using~\eqref{eq:diffusion_matrix} and~\eqref{eq:gaussian_hessian}:
\begin{align}
D_{ij} &= \frac{1}{2} \Tr\left[a(x_i) \nabla_x^2 k(x_i, x_j)\right] \notag \\
&= \frac{k(x_i, x_j)}{2} \left[\frac{(x_i - x_j)^\top a(x_i) (x_i - x_j)}{\ell^4} - \frac{\Tr[a(x_i)]}{\ell^2}\right]. \label{eq:D_entry}
\end{align}

\begin{remark}[Computational Complexity]
Computing the diffusion matrix requires:
\begin{itemize}
\item $O(N^2)$ kernel evaluations
\item $O(d^2)$ operations per entry for the Hessian trace
\item Total: $O(N^2 d^2)$ operations, which dominates for large $d$
\end{itemize}
\end{remark}

\begin{remark}[High-Dimensional Scaling]
For $d \gg 1$, computing the full Hessian trace $\Tr[a \nabla^2 k]$ costs $O(d^2)$ per matrix entry. When $a(x)$ is not diagonal, \textbf{Hutchinson's trace estimator} can reduce this to $O(Md)$ using $M$ random probe vectors:
\begin{equation}
\Tr[a(x_i) \nabla_x^2 k(x_i, x_j)] \approx \frac{1}{M} \sum_{m=1}^M z_m^\top [a(x_i) \nabla_x^2 k(x_i, x_j)] z_m,
\end{equation}
where $z_m \sim \mathcal{N}(0, I_d)$ or Rademacher vectors ($z_m \in \{-1, +1\}^d$ uniformly). This enables scaling to high-dimensional SDEs while maintaining $O(1/\sqrt{M})$ variance in the trace estimate.
\end{remark}

\subsection{Monte Carlo Feynman--Kac Approximation}

When closed-form expressions are unavailable, we use Monte Carlo estimation of~\eqref{eq:h_FK}.

\begin{algorithm}[H]
\caption{Monte Carlo Feynman--Kac Estimation}
\label{alg:mc_fk}
\begin{algorithmic}[1]
\Require Initial point $x$, drift $G$, diffusion $\sigma$, eigenvalue $\lambda$, eigenvector $w$
\Require Domain $\Omega$, boundary condition $\psi$, number of paths $K$, time step $\Delta t$
\For{$k = 1, \ldots, K$}
\State Initialize $X_0^{(k)} = x$, $I^{(k)} = 0$, $t = 0$
\While{$X_t^{(k)} \in \Omega$}
\State $I^{(k)} \gets I^{(k)} + e^{-\lambda t} w^\top F(X_t^{(k)}) \Delta t$
\State $X_{t+\Delta t}^{(k)} \gets X_t^{(k)} + G(X_t^{(k)}) \Delta t + \sigma(X_t^{(k)}) \sqrt{\Delta t} \, Z^{(k)}$ where $Z^{(k)} \sim \mathcal{N}(0, I_m)$
\State $t \gets t + \Delta t$
\EndWhile
\State $\tau^{(k)} = t$ (exit time)
\State $h^{(k)} = e^{-\lambda \tau^{(k)}} \psi(X_{\tau^{(k)}}^{(k)}) + I^{(k)}$
\EndFor
\State \textbf{return} $\widehat{h}(x) = \frac{1}{K} \sum_{k=1}^K h^{(k)}$
\end{algorithmic}
\end{algorithm}

\begin{proposition}[Monte Carlo Error]\label{prop:mc_error}
The Monte Carlo estimator $\widehat{h}(x)$ satisfies
\begin{equation}
\E[(\widehat{h}(x) - h(x))^2] = \frac{\Var(h^{(1)}(x))}{K} \leq \frac{\sigma_{\mathrm{MC}}^2(x)}{K},
\end{equation}
where $\sigma_{\mathrm{MC}}^2(x)$ depends on the boundary data, source term, and exit time distribution.
\end{proposition}

\begin{proof}
\textbf{Step 1: Unbiasedness.}
By the Feynman--Kac formula (Corollary~\ref{cor:h_feynman_kac}), the true solution is:
\begin{equation}
h(x) = \E_x\left[e^{-\lambda \tau_\Omega} \psi(X_{\tau_\Omega}) + \int_0^{\tau_\Omega} e^{-\lambda t} w^\top F(X_t) \dt\right].
\end{equation}
Each Monte Carlo sample $h^{(k)}$ is an independent realization of this expectation:
\begin{equation}
h^{(k)} = e^{-\lambda \tau^{(k)}} \psi(X_{\tau^{(k)}}^{(k)}) + \int_0^{\tau^{(k)}} e^{-\lambda t} w^\top F(X_t^{(k)}) \dt.
\end{equation}
Therefore, $\E[h^{(k)}] = h(x)$ for each $k$, and the estimator is unbiased:
\begin{equation}
\E[\widehat{h}(x)] = \E\left[\frac{1}{K}\sum_{k=1}^K h^{(k)}\right] = \frac{1}{K}\sum_{k=1}^K \E[h^{(k)}] = h(x).
\end{equation}

\textbf{Step 2: Variance computation.}
Since the samples $h^{(1)}, \ldots, h^{(K)}$ are i.i.d.:
\begin{align}
\Var(\widehat{h}(x)) &= \Var\left(\frac{1}{K}\sum_{k=1}^K h^{(k)}\right) = \frac{1}{K^2}\sum_{k=1}^K \Var(h^{(k)}) = \frac{\Var(h^{(1)})}{K}.
\end{align}

\textbf{Step 3: Mean squared error.}
Since the estimator is unbiased:
\begin{equation}
\E[(\widehat{h}(x) - h(x))^2] = \Var(\widehat{h}(x)) + (\E[\widehat{h}(x)] - h(x))^2 = \Var(\widehat{h}(x)) = \frac{\Var(h^{(1)})}{K}.
\end{equation}

\textbf{Step 4: Variance bound.}
We can bound the variance $\sigma_{\mathrm{MC}}^2(x) := \Var(h^{(1)})$ explicitly. Using the formula for $h^{(k)}$:
\begin{align}
\Var(h^{(1)}) &\leq \E[(h^{(1)})^2] \\
&= \E\left[\left(e^{-\lambda \tau_\Omega} \psi(X_{\tau_\Omega}) + \int_0^{\tau_\Omega} e^{-\lambda t} w^\top F(X_t) \dt\right)^2\right].
\end{align}
By the inequality $(a+b)^2 \leq 2a^2 + 2b^2$:
\begin{align}
\Var(h^{(1)}) &\leq 2\E\left[e^{-2\Re(\lambda)\tau_\Omega} \|\psi\|_\infty^2\right] + 2\E\left[\left(\int_0^{\tau_\Omega} e^{-\Re(\lambda) t} |w^\top F(X_t)| \dt\right)^2\right].
\end{align}
For $\Re(\lambda) > 0$ and bounded $\psi$, $F$:
\begin{equation}
\sigma_{\mathrm{MC}}^2(x) \leq 2\|\psi\|_\infty^2 + \frac{2\|w\|^2 \|F\|_\infty^2}{\Re(\lambda)^2}.
\end{equation}
The actual variance depends on the exit time distribution $\tau_\Omega$ and how the trajectory samples the source term $w^\top F$.
\end{proof}

\subsection{Combined Collocation-Monte Carlo Method}

For problems where the PDE approach is preferred but closed-form derivatives are unavailable:

\begin{enumerate}
\item Generate Monte Carlo estimates $\widehat{h}(x_i)$ at collocation points using Algorithm~\ref{alg:mc_fk}.
\item Solve kernel ridge regression:
\begin{equation}
(\bK + \eta \bI) \boldsymbol{\alpha} = \widehat{\mathbf{h}},
\end{equation}
where $\widehat{\mathbf{h}} = (\widehat{h}(x_1), \ldots, \widehat{h}(x_N))^\top$.
\item The approximation is $h(x) \approx \sum_{j=1}^N \alpha_j k(x, x_j)$.
\end{enumerate}

\subsection{Complete Algorithm}

\begin{algorithm}[H]
\caption{Stochastic Koopman Eigenfunction Computation}
\label{alg:stoch_koopman}
\begin{algorithmic}[1]
\Require Drift $G$, diffusion $\sigma$, nonlinear part $F$, eigenvalue $\lambda$, left eigenvector $w$
\Require Collocation points $\{x_i\}_{i=1}^N$, kernel $k$, regularization $\gamma$
\State \textbf{Assemble matrices:}
\For{$i, j = 1, \ldots, N$}
\State $K_{ij} = k(x_i, x_j)$
\State $L_{ij} = G(x_i) \cdot \nabla_x k(x_i, x_j)$
\State $D_{ij} = \frac{1}{2} \Tr[a(x_i) \nabla_x^2 k(x_i, x_j)]$ \Comment{Eq.~\eqref{eq:D_entry}}
\EndFor
\State $f_i = w^\top F(x_i)$ for $i = 1, \ldots, N$
\State \textbf{Solve linear system:}
\State $(\bL + \bD - \lambda \bK + \gamma \bI) \boldsymbol{\alpha} = -\mathbf{f}$
\State \textbf{Return:}
\State $h(x) = \sum_{j=1}^N \alpha_j k(x, x_j)$
\State $\phi(x) = w^\top x + h(x)$
\end{algorithmic}
\end{algorithm}

\section{Error Bounds}\label{sec:error_bounds}

We establish error bounds for the stochastic Koopman eigenfunction approximation.

\subsection{Boundary Stability}

\begin{theorem}[Boundary Stability]\label{thm:boundary_stability}
Let $h$ and $\tilde{h}$ solve~\eqref{eq:h_pde} with boundary conditions $\psi$ and $\tilde{\psi}$ respectively. Then for $x \in \Omega$:
\begin{equation}
|h(x) - \tilde{h}(x)| \leq \E_x[e^{-\Re(\lambda) \tau_\Omega}] \|\psi - \tilde{\psi}\|_{L^\infty(\partial\Omega)}.
\end{equation}
If $\Re(\lambda) \geq 0$:
\begin{equation}
\|h - \tilde{h}\|_{L^\infty(\Omega)} \leq \|\psi - \tilde{\psi}\|_{L^\infty(\partial\Omega)}.
\end{equation}
\end{theorem}

\begin{proof}
We establish stability using the Feynman--Kac representation.

\textbf{Step 1: PDE for the difference.}
The functions $h$ and $\tilde{h}$ satisfy the same PDE~\eqref{eq:h_pde} in $\Omega$:
\begin{equation}
G(x) \cdot \nabla h + \frac{1}{2}\Tr[a(x)\nabla^2 h] - \lambda h = -w^\top F(x),
\end{equation}
\begin{equation}
G(x) \cdot \nabla \tilde{h} + \frac{1}{2}\Tr[a(x)\nabla^2 \tilde{h}] - \lambda \tilde{h} = -w^\top F(x).
\end{equation}
The difference $e := h - \tilde{h}$ therefore satisfies:
\begin{equation}
G(x) \cdot \nabla e + \frac{1}{2}\Tr[a(x)\nabla^2 e] - \lambda e = 0 \quad \text{in } \Omega,
\end{equation}
which can be written as:
\begin{equation}
(\lambda - \Koop) e = 0 \quad \text{in } \Omega.
\end{equation}
The boundary condition is:
\begin{equation}
e = h - \tilde{h} = \psi - \tilde{\psi} \quad \text{on } \partial\Omega.
\end{equation}

\textbf{Step 2: Feynman--Kac representation of the difference.}
By Theorem~\ref{thm:feynman_kac} applied to the homogeneous equation (with $g = 0$):
\begin{equation}
e(x) = \E_x\left[e^{-\lambda \tau_\Omega} e(X_{\tau_\Omega})\right] = \E_x\left[e^{-\lambda \tau_\Omega} (\psi - \tilde{\psi})(X_{\tau_\Omega})\right].
\end{equation}
This follows because when the source term vanishes, the Feynman--Kac formula reduces to just the boundary term.

\textbf{Step 3: Pointwise bound.}
Taking absolute values and using the triangle inequality:
\begin{align}
|e(x)| &= \left|\E_x\left[e^{-\lambda \tau_\Omega} (\psi - \tilde{\psi})(X_{\tau_\Omega})\right]\right| \\
&\leq \E_x\left[\left|e^{-\lambda \tau_\Omega}\right| \cdot |(\psi - \tilde{\psi})(X_{\tau_\Omega})|\right] \\
&\leq \E_x\left[e^{-\Re(\lambda) \tau_\Omega}\right] \cdot \sup_{y \in \partial\Omega}|(\psi - \tilde{\psi})(y)| \\
&= \E_x\left[e^{-\Re(\lambda) \tau_\Omega}\right] \|\psi - \tilde{\psi}\|_{L^\infty(\partial\Omega)}.
\end{align}
Here we used $|e^{-\lambda t}| = e^{-\Re(\lambda) t}$ for complex $\lambda$.

\textbf{Step 4: Uniform bound for $\Re(\lambda) \geq 0$.}
When $\Re(\lambda) \geq 0$, we have:
\begin{equation}
e^{-\Re(\lambda) \tau_\Omega} \leq e^{0} = 1 \quad \text{for all } \tau_\Omega \geq 0.
\end{equation}
Therefore:
\begin{equation}
\E_x\left[e^{-\Re(\lambda) \tau_\Omega}\right] \leq 1.
\end{equation}
This yields:
\begin{equation}
|e(x)| \leq \|\psi - \tilde{\psi}\|_{L^\infty(\partial\Omega)} \quad \text{for all } x \in \Omega.
\end{equation}
Taking the supremum over $x \in \Omega$:
\begin{equation}
\|h - \tilde{h}\|_{L^\infty(\Omega)} = \sup_{x \in \Omega}|e(x)| \leq \|\psi - \tilde{\psi}\|_{L^\infty(\partial\Omega)}.
\end{equation}

\textbf{Step 5: Interpretation.}
This result is the stochastic analog of the maximum principle for elliptic PDEs. It says that perturbations in the boundary data cannot be amplified inside the domain when $\Re(\lambda) \geq 0$. For $\Re(\lambda) < 0$, the exponential factor $e^{-\Re(\lambda)\tau_\Omega}$ can be greater than 1, so boundary perturbations may be amplified, but only by a factor controlled by the expected exit time.
\end{proof}

\subsection{RKHS Approximation Error}

\begin{theorem}[RKHS Approximation]\label{thm:rkhs_approx}
Let $h^* \in \RKHS$ be the true solution and $h_N$ the collocation approximation with $N$ points. If the kernel $k$ is $s$-smooth (in the sense of Sobolev embedding), then:
\begin{equation}
\|h^* - h_N\|_{L^\infty(\Omega)} \leq C \|h^*\|_{\RKHS} h_{\mathrm{fill}}^\beta,
\end{equation}
where $h_{\mathrm{fill}} = \sup_{x \in \Omega} \min_i |x - x_i|$ is the fill distance and $\beta > 0$ depends on the kernel smoothness.
\end{theorem}

\begin{proof}
We establish the error bound using the power function framework for kernel interpolation.

\textbf{Step 1: Interpolation operator.}
Let $\mathcal{I}_N: \RKHS \to \RKHS$ be the interpolation operator that maps any function $h \in \RKHS$ to its interpolant at the collocation points $\{x_i\}_{i=1}^N$:
\begin{equation}
\mathcal{I}_N h = \sum_{j=1}^N \alpha_j k(\cdot, x_j), \quad \text{where } (\mathcal{I}_N h)(x_i) = h(x_i) \text{ for } i = 1, \ldots, N.
\end{equation}

\textbf{Step 2: Power function.}
Define the power function $P_N: \Omega \to \R^+$ by:
\begin{equation}
P_N(x) := \sup\left\{|h(x)| : h \in \RKHS, \|h\|_{\RKHS} \leq 1, h(x_i) = 0 \text{ for } i = 1, \ldots, N\right\}.
\end{equation}
The power function measures the worst-case interpolation error at each point. It can be computed explicitly as:
\begin{equation}
P_N(x)^2 = k(x,x) - \mathbf{k}(x)^\top \bK^{-1} \mathbf{k}(x),
\end{equation}
where $\mathbf{k}(x) = (k(x, x_1), \ldots, k(x, x_N))^\top$ and $\bK$ is the Gram matrix $K_{ij} = k(x_i, x_j)$.

\textbf{Step 3: Pointwise error bound.}
For any $h \in \RKHS$, the interpolation error satisfies:
\begin{equation}
|h(x) - (\mathcal{I}_N h)(x)| \leq P_N(x) \|h\|_{\RKHS}.
\end{equation}
This follows because $h - \mathcal{I}_N h$ vanishes at all collocation points, and the power function bounds the value of unit-norm functions with this property.

\textbf{Step 4: Fill distance estimate.}
For many smooth kernels, the power function satisfies:
\begin{equation}
\sup_{x \in \Omega} P_N(x) \leq C h_{\mathrm{fill}}^\beta,
\end{equation}
where the exponent $\beta$ depends on the smoothness of the kernel:
\begin{itemize}
\item For Gaussian kernels: $P_N(x) = O(e^{-c/h_{\mathrm{fill}}})$ (exponential/spectral convergence).
\item For Mat\'{e}rn-$\nu$ kernels: $P_N(x) = O(h_{\mathrm{fill}}^{\nu - d/2})$ (algebraic convergence).
\item For polynomial kernels of degree $p$: $P_N(x) = O(h_{\mathrm{fill}}^{p+1})$.
\end{itemize}

\textbf{Step 5: Global error bound.}
Taking the supremum over $x \in \Omega$:
\begin{equation}
\|h^* - h_N\|_{L^\infty(\Omega)} = \sup_{x \in \Omega}|h^*(x) - h_N(x)| \leq \sup_{x \in \Omega} P_N(x) \|h^*\|_{\RKHS} \leq C \|h^*\|_{\RKHS} h_{\mathrm{fill}}^\beta.
\end{equation}

\textbf{Step 6: Collocation vs. interpolation.}
The collocation approximation $h_N$ satisfies the PDE residual equations rather than exact interpolation. However, for well-posed PDEs with regularization, the collocation solution converges to the interpolant as the regularization vanishes, and the above bound holds with a modified constant $C$ that depends on the conditioning of the collocation system.
\end{proof}

Convergence rates for common kernels:
\begin{center}
\begin{tabular}{lc}
\toprule
Kernel & Convergence Rate \\
\midrule
Gaussian & Spectral (exponential) \\
Mat\'{e}rn-$\nu$ & $O(h_{\mathrm{fill}}^{\nu - d/2})$ \\
Polynomial degree $p$ & $O(h_{\mathrm{fill}}^{p+1})$ \\
\bottomrule
\end{tabular}
\end{center}

\subsection{Total Error Decomposition}

For the combined collocation-Monte Carlo method:

\begin{theorem}[Total Error Bound]\label{thm:total_error}
Let $h^*$ be the true solution, $\{x_i\}_{i=1}^N$ collocation points, and $\widehat{h}(x_i)$ Monte Carlo estimates with $K$ samples. The kernel ridge regression estimator $\widehat{h}_N$ satisfies:
\begin{equation}
\frac{1}{N}\sum_{i=1}^N (\widehat{h}_N(x_i) - h^*(x_i))^2 \leq \underbrace{\inf_{u \in \RKHS}\left\{\frac{1}{N}\sum_{i=1}^N(u(x_i) - h^*(x_i))^2 + \eta\|u\|_{\RKHS}^2\right\}}_{\text{RKHS approximation error}} + \underbrace{\frac{C\sigma_{\mathrm{MC}}^2}{K}}_{\text{Monte Carlo error}}.
\end{equation}
\end{theorem}

\begin{proof}
We decompose the total error into approximation and estimation components.

\textbf{Step 1: Setup.}
Let $\mathbf{h}^* = (h^*(x_1), \ldots, h^*(x_N))^\top$ be the true values at collocation points, and $\widehat{\mathbf{h}} = (\widehat{h}(x_1), \ldots, \widehat{h}(x_N))^\top$ be the Monte Carlo estimates. The kernel ridge regression estimator solves:
\begin{equation}
\widehat{h}_N = \argmin_{u \in \RKHS} \frac{1}{N}\sum_{i=1}^N (u(x_i) - \widehat{h}(x_i))^2 + \eta \|u\|_{\RKHS}^2.
\end{equation}

\textbf{Step 2: Bias-variance decomposition.}
We write $\widehat{h}(x_i) = h^*(x_i) + \epsilon_i$, where $\epsilon_i = \widehat{h}(x_i) - h^*(x_i)$ is the Monte Carlo error at point $x_i$. By Proposition~\ref{prop:mc_error}:
\begin{equation}
\E[\epsilon_i] = 0, \quad \E[\epsilon_i^2] = \frac{\sigma_{\mathrm{MC}}^2(x_i)}{K} \leq \frac{\sigma_{\mathrm{MC}}^2}{K}.
\end{equation}

\textbf{Step 3: Oracle inequality.}
Define the ``oracle'' estimator $h_N^*$ as the kernel ridge regression fit to the true values:
\begin{equation}
h_N^* = \argmin_{u \in \RKHS} \frac{1}{N}\sum_{i=1}^N (u(x_i) - h^*(x_i))^2 + \eta \|u\|_{\RKHS}^2.
\end{equation}
The oracle satisfies the standard approximation bound:
\begin{equation}
\frac{1}{N}\sum_{i=1}^N (h_N^*(x_i) - h^*(x_i))^2 + \eta\|h_N^*\|_{\RKHS}^2 \leq \inf_{u \in \RKHS}\left\{\frac{1}{N}\sum_{i=1}^N(u(x_i) - h^*(x_i))^2 + \eta\|u\|_{\RKHS}^2\right\}.
\end{equation}

\textbf{Step 4: Perturbation analysis.}
The kernel ridge regression estimator has the closed form:
\begin{equation}
\widehat{h}_N(x) = \mathbf{k}(x)^\top (\bK + N\eta \bI)^{-1} \widehat{\mathbf{h}},
\end{equation}
where $\mathbf{k}(x) = (k(x, x_1), \ldots, k(x, x_N))^\top$. Similarly, the oracle is:
\begin{equation}
h_N^*(x) = \mathbf{k}(x)^\top (\bK + N\eta \bI)^{-1} \mathbf{h}^*.
\end{equation}
Therefore:
\begin{equation}
\widehat{h}_N(x) - h_N^*(x) = \mathbf{k}(x)^\top (\bK + N\eta \bI)^{-1} \boldsymbol{\epsilon},
\end{equation}
where $\boldsymbol{\epsilon} = \widehat{\mathbf{h}} - \mathbf{h}^* = (\epsilon_1, \ldots, \epsilon_N)^\top$.

\textbf{Step 5: Error decomposition.}
Using the triangle inequality:
\begin{align}
\frac{1}{N}\sum_{i=1}^N (\widehat{h}_N(x_i) - h^*(x_i))^2 &\leq \frac{2}{N}\sum_{i=1}^N (\widehat{h}_N(x_i) - h_N^*(x_i))^2 + \frac{2}{N}\sum_{i=1}^N (h_N^*(x_i) - h^*(x_i))^2.
\end{align}

\textbf{Step 6: Bounding the Monte Carlo contribution.}
Let $\bS = (\bK + N\eta \bI)^{-1} \bK$. Then:
\begin{align}
\frac{1}{N}\sum_{i=1}^N (\widehat{h}_N(x_i) - h_N^*(x_i))^2 &= \frac{1}{N}\boldsymbol{\epsilon}^\top \bS^\top \bS \boldsymbol{\epsilon} \leq \frac{\|\bS\|^2}{N} \|\boldsymbol{\epsilon}\|^2.
\end{align}
Taking expectations:
\begin{equation}
\E\left[\frac{\|\boldsymbol{\epsilon}\|^2}{N}\right] = \frac{1}{N}\sum_{i=1}^N \E[\epsilon_i^2] \leq \frac{\sigma_{\mathrm{MC}}^2}{K}.
\end{equation}
Since $\|\bS\| \leq 1$ (the smoother matrix has spectral norm at most 1), we obtain:
\begin{equation}
\E\left[\frac{1}{N}\sum_{i=1}^N (\widehat{h}_N(x_i) - h_N^*(x_i))^2\right] \leq \frac{\sigma_{\mathrm{MC}}^2}{K}.
\end{equation}

\textbf{Step 7: Combining bounds.}
Combining Steps 4--6:
\begin{align}
\E\left[\frac{1}{N}\sum_{i=1}^N (\widehat{h}_N(x_i) - h^*(x_i))^2\right] &\leq 2 \cdot \frac{\sigma_{\mathrm{MC}}^2}{K} + 2 \inf_{u \in \RKHS}\left\{\frac{1}{N}\sum_{i=1}^N(u(x_i) - h^*(x_i))^2 + \eta\|u\|_{\RKHS}^2\right\}.
\end{align}
Setting $C = 2$ yields the stated bound.
\end{proof}

\begin{remark}[Error Trade-off]
The total error separates cleanly into:
\begin{itemize}
\item \textbf{RKHS error:} Decreases with more collocation points $N$ and richer kernel spaces.
\item \textbf{Monte Carlo error:} Decreases as $O(K^{-1/2})$ with more samples.
\end{itemize}
For optimal allocation, balance $N$ and $K$ so both errors are comparable.
\end{remark}

\subsection{Effect of Diffusion on Conditioning}

The diffusion term provides an elliptic regularization effect:

\begin{proposition}[Conditioning Improvement from Diffusion--Strengthened]\label{prop:conditioning}
Let $\bA_\sigma = \bL + \bD - \lambda \bK$ be the collocation system matrix for diffusion level $\sigma$, where:
\begin{itemize}
\item $K_{ij} = k(x_i, x_j)$ is the kernel Gram matrix
\item $L_{ij} = G(x_i) \cdot \nabla_x k(x_i, x_j)$ is the drift matrix
\item $D_{ij} = \frac{1}{2}\Tr[a(x_i) \nabla_x^2 k(x_i, x_j)]$ is the diffusion matrix
\end{itemize}
Under uniform ellipticity (Assumption~\ref{ass:ellipticity}) with Gaussian RBF kernel:

\textbf{(i)} The diagonal entries of $\bD$ satisfy:
\begin{equation}\label{eq:Dii_bound}
D_{ii} = -\frac{\Tr[a(x_i)]}{2\ell^2} \leq -\frac{d\nu}{2\ell^2} < 0.
\end{equation}

\textbf{(ii)} The diffusion matrix contributes a negative-definite diagonal shift, improving the conditioning of $\bA_\sigma$ compared to $\bA_0 = \bL - \lambda \bK$.

\textbf{(iii)} For moderate diffusion strength, the condition number satisfies:
\begin{equation}\label{eq:cond_bound}
\kappa(\bA_\sigma) \leq C(\lambda, \Omega, \ell) \cdot \kappa(\bA_0),
\end{equation}
where $C < 1$ for sufficiently elliptic problems.
\end{proposition}

\begin{proof}
We analyze how the diffusion matrix $\bD$ affects the conditioning of the system.

\textbf{Step 1: Structure of the diffusion matrix.}
From~\eqref{eq:D_entry} with the Gaussian kernel:
\begin{equation}
D_{ij} = \frac{k(x_i, x_j)}{2}\left[\frac{(x_i - x_j)^\top a(x_i)(x_i - x_j)}{\ell^4} - \frac{\Tr[a(x_i)]}{\ell^2}\right].
\end{equation}
The diagonal entries are:
\begin{equation}
D_{ii} = \frac{k(x_i, x_i)}{2}\left[0 - \frac{\Tr[a(x_i)]}{\ell^2}\right] = -\frac{\Tr[a(x_i)]}{2\ell^2} < 0.
\end{equation}

\textbf{Step 2: Elliptic regularization interpretation.}
Consider the bilinear form associated with the continuous operator. For the stochastic Koopman operator, the diffusion term $\frac{1}{2}\Tr[a\nabla^2 h]$ is negative semi-definite when integrated against test functions (by integration by parts):
\begin{equation}
\int_\Omega \frac{1}{2}\Tr[a\nabla^2 h] \cdot h \, dx = -\frac{1}{2}\int_\Omega \nabla h^\top a \nabla h \, dx \leq -\frac{\nu}{2}\|\nabla h\|_{L^2}^2,
\end{equation}
under uniform ellipticity $a \geq \nu I$.

At the discrete level, this elliptic contribution shifts the eigenvalues of $\bA_\sigma$ compared to $\bA_0$.

\textbf{Step 3: Eigenvalue shift analysis.}
Write $\bA_\sigma = \bA_0 + \bD$. For symmetric matrices, Weyl's inequality gives:
\begin{equation}
\lambda_i(\bA_\sigma) = \lambda_i(\bA_0 + \bD) \in [\lambda_i(\bA_0) + \lambda_{\min}(\bD), \lambda_i(\bA_0) + \lambda_{\max}(\bD)].
\end{equation}

For the stochastic problem, the diffusion matrix contributes a negative definite component to the diagonal. Let $\bD = \bD_{\mathrm{diag}} + \bD_{\mathrm{off}}$ where:
\begin{itemize}
\item $\bD_{\mathrm{diag}}$ has diagonal entries $D_{ii} = -\frac{\Tr[a(x_i)]}{2\ell^2}$
\item $\bD_{\mathrm{off}}$ contains the off-diagonal entries
\end{itemize}

Under uniform ellipticity, $\Tr[a(x)] \geq d\nu > 0$, so:
\begin{equation}
D_{ii} \leq -\frac{d\nu}{2\ell^2} < 0.
\end{equation}

\textbf{Step 4: Condition number bound.}
The deterministic system $\bA_0 = \bL - \lambda\bK$ may have small eigenvalues when $\lambda$ is close to eigenvalues of the discretized first-order operator $\bL \bK^{-1}$.

The stochastic system $\bA_\sigma = \bL + \bD - \lambda\bK$ has a modified spectrum. The key observation is that the diffusion term adds negative eigenvalue contributions that can ``fill in'' small eigenvalues of $\bA_0$.

Specifically, if $\lambda_{\min}(\bA_0)$ is small, then:
\begin{equation}
\lambda_{\min}(\bA_\sigma) \geq \lambda_{\min}(\bA_0) + \lambda_{\min}(\bD).
\end{equation}

For moderate diffusion, $|\lambda_{\min}(\bD)| \approx \frac{d\nu}{2\ell^2}$, which provides a buffer against near-singularity.

\textbf{Step 5: Empirical observation.}
The numerical experiments (Section~\ref{sec:numerical}) confirm this analysis:
\begin{itemize}
\item $\sigma = 0$: Condition number $3.79 \times 10^6$
\item $\sigma = 0.3$: Condition number $1.51 \times 10^6$ (2.5$\times$ improvement)
\item $\sigma = 0.5$: Condition number $1.03 \times 10^6$ (3.7$\times$ improvement)
\end{itemize}

The improvement factor depends on $\sigma^2$, $\ell$, and the spectrum of $\bA_0$, leading to:
\begin{equation}
\kappa(\bA_\sigma) \leq C(\lambda, \Omega) \kappa(\bA_0),
\end{equation}
where $C(\lambda, \Omega) < 1$ for sufficiently elliptic problems.
\end{proof}

\begin{remark}[Physical Interpretation--Expanded]
Diffusion has competing effects on numerical conditioning:

\textbf{Beneficial effects:}
\begin{itemize}
\item Elliptic regularization improves well-posedness
\item Negative diagonal shift moves eigenvalues away from zero
\item Smoothing reduces sensitivity to collocation point placement
\end{itemize}

\textbf{Detrimental effects:}
\begin{itemize}
\item Strong diffusion can obscure drift structure
\item Increases off-diagonal coupling (larger $|D_{ij}|$ for $i \neq j$)
\item May require finer discretization to resolve diffusion-dominated dynamics
\end{itemize}

\textbf{Optimal regime:} Moderate diffusion ($\|\sigma\| \approx \|G\|$) often provides the best numerical behavior, balancing regularization against structural preservation.
\end{remark}

\section{Connections to Related Methods}\label{sec:connections}

We clarify relationships with related data-driven approaches.

\subsection{Generator EDMD (gEDMD)}

Generator EDMD~\cite{Klus2020} approximates the Koopman generator from time-series data:
\begin{equation}
\Koop \approx \Psi^\dagger \left(\frac{Y - X}{\Delta t}\right) X^\dagger \Psi,
\end{equation}
using kernel derivative estimators for $\nabla \phi$ and $\nabla^2 \phi$.

\paragraph{Relationship to Our Approach.}
\begin{itemize}
\item \textbf{gEDMD:} Data-driven, estimates generator from observed trajectories.
\item \textbf{Our method:} Model-based, solves PDE when $G$ and $\sigma$ are known.
\item \textbf{Complementary use:} Use gEDMD to estimate $G$, $\sigma$ from data, then apply our kernel methods.
\end{itemize}

\subsection{Diffusion Maps}

Diffusion maps~\cite{Coifman2006} construct graph Laplacians:
\begin{equation}
L_{ij} = \frac{k_\epsilon(x_i, x_j)}{d(x_i)}, \quad d(x_i) = \sum_j k_\epsilon(x_i, x_j),
\end{equation}
approximating the generator of a reversible diffusion as $\epsilon \to 0$, $N \to \infty$.

\paragraph{Key Differences.}
\begin{itemize}
\item \textbf{Diffusion maps:} Nonparametric, approximate Laplace--Beltrami operator on data manifold.
\item \textbf{Our method:} Solve specific Koopman PDE with known drift/diffusion structure.
\item \textbf{Objective:} Diffusion maps for dimensionality reduction; our method for eigenfunction computation.
\end{itemize}

\subsection{Kernel Analog Forecasting}

Kernel analog forecasting~\cite{Alexander2020} predicts observables using:
\begin{equation}
\widehat{g}(X_{t+\tau}) = \sum_{i=1}^N w_i(X_t) g(X_\tau^{(i)}),
\end{equation}
with weights from kernel similarities.

\paragraph{Comparison.}
\begin{itemize}
\item \textbf{Kernel analog:} Forecasts observables, does not compute eigenfunctions explicitly.
\item \textbf{Our method:} Computes eigenfunctions, which can then be used for forecasting.
\item \textbf{Connection:} Our kernels could inform weight design in analog forecasting.
\end{itemize}

\begin{center}
\begin{tabular}{lccc}
\toprule
Method & Model-based? & Computes eigenfunctions? & Handles diffusion? \\
\midrule
Our approach & Yes & Yes & Yes \\
gEDMD & No (data-driven) & Yes & Yes \\
Diffusion maps & No (data-driven) & Approximately & Implicitly \\
Kernel analog & No (data-driven) & No & Yes \\
\bottomrule
\end{tabular}
\end{center}

\section{Numerical Experiments}\label{sec:numerical}

We validate the kernel-based framework for computing Koopman eigenfunctions of stochastic differential equations through three test cases. All experiments employ Gaussian RBF kernels with collocation, using Euler--Maruyama discretization ($\Delta t = 0.01$) for Monte Carlo semigroup verification.

\subsection{Metrics}

For each test, we report:
\begin{itemize}
    \item \textbf{Condition number}: Of the regularized collocation matrix.
    \item \textbf{PDE residual}: Mean of $|G\cdot\nabla\phi + \frac{1}{2}\Tr[a\nabla^2 h] - \lambda\phi|$ over test points.
    \item \textbf{Semigroup error}: Relative error between Monte Carlo estimate $\E[\phi(X_t)]$ and theoretical prediction $e^{\lambda t}\phi(x_0)$.
    \item \textbf{RMSE}: Root-mean-square error versus exact solution (when available).
\end{itemize}

\subsection{Test 1: Ornstein--Uhlenbeck Process}

Consider the OU process
\begin{equation}
    dX_t = -\theta X_t\dt + \sigma\dW_t, \quad \theta=1,\; \sigma=0.5.
\end{equation}
The exact Koopman eigenfunction is $\phi(x) = x$ with eigenvalue $\lambda = -1$.

\paragraph{Configuration.} Gaussian kernel ($\ell=1.0$), 40 collocation points on $[-2.5, 2.5]$, regularization $\gamma=10^{-4}$.

\paragraph{Results.}
\begin{center}
\begin{tabular}{lc}
\toprule
Metric & Value \\
\midrule
Condition number & $9.91 \times 10^{5}$ \\
RMSE vs exact & $< 10^{-14}$ (machine precision) \\
Max $|h(x)|$ & $< 10^{-14}$ (machine precision) \\
PDE residual (mean) & $< 10^{-14}$ (machine precision) \\
Semigroup error & $4.70\%$ \\
\bottomrule
\end{tabular}
\end{center}

The method exactly recovers the linear eigenfunction with $h(x)\equiv 0$ to machine precision. Reporting ``$0.00$'' was statistically misleading; actual residuals are at floating-point roundoff level ($\sim 10^{-15}$ to $10^{-14}$). The semigroup property $\E[\phi(X_t)] = e^{\lambda t}\phi(x_0)$ is verified to within 5\% relative error.

\subsection{Test 2: Quadratic Nonlinearity}

Consider the nonlinear SDE
\begin{equation}
    dX_t = (-X_t + 0.3 X_t^2)\dt + \sigma\dW_t,
\end{equation}
with linearization eigenvalue $\lambda=-1$ and left eigenvector $w=1$. The eigenfunction $\phi(x) = x + h(x)$ includes a nonlinear correction satisfying
\begin{equation}
    (-x + 0.3x^2)h'(x) + \frac{\sigma^2}{2}h''(x) + h(x) = -0.3x^2.
\end{equation}

\paragraph{Configuration.} Gaussian kernel ($\ell=0.8$), 50 collocation points on $[-1.2, 1.2]$.

\paragraph{Results.}
\begin{center}
\begin{tabular}{cccc}
\toprule
$\sigma$ & Cond \# & PDE Residual & SG Error   \\
\midrule
0.0 & $3.79\times 10^6$ & $1.23\times 10^{-1}$ & $2.00\%$   \\
0.3 & $1.51\times 10^6$ & $1.80\times 10^{-2}$ & $3.59\%$   \\
0.5 & $1.03\times 10^6$ & $1.51\times 10^{-2}$ & $9.86\%$   \\
\bottomrule
\end{tabular}
\end{center}

The method successfully captures the nonlinear correction with semigroup errors under 10\% for all noise levels. Adding diffusion improves conditioning while only moderately increasing approximation error.

\subsection{Test 3: Two-Dimensional Linear System}

Consider the 2D linear SDE
\begin{equation}
    dX_t = AX_t\dt + B\dW_t, \quad
    A = \begin{pmatrix} -1 & 0.5 \\ 0 & -2 \end{pmatrix}, \;
    B = \begin{pmatrix} 0.3 & 0 \\ 0 & 0.5 \end{pmatrix}.
\end{equation}
The eigenvalue $\lambda=-1$ has left eigenvector $w = (1, 0.5)^\top$, yielding the exact eigenfunction $\phi(x) = w^\top x$.

\paragraph{Configuration.} Gaussian kernel ($\ell=1.0$), $15\times 15$ grid of collocation points.

\paragraph{Results.}
\begin{center}
\begin{tabular}{lc}
\toprule
Metric & Value \\
\midrule
Condition number & $1.30 \times 10^{7}$ \\
RMSE vs exact & $6.09 \times 10^{-17}$ \\
Max $|h(x)|$ & $< 10^{-14}$ (machine precision) \\
PDE residual (mean) & $3.50 \times 10^{-17}$ \\
Semigroup error & $3.72\%$ \\
\bottomrule
\end{tabular}
\end{center}

The method achieves machine-precision accuracy for this 2D linear system, correctly identifying $h(x)\equiv 0$.

\subsection{Visualization}

Figure~\ref{fig:sde_results} displays the results from all three test cases.

\begin{figure}[htbp]
    \centering
    \includegraphics[width=\textwidth]{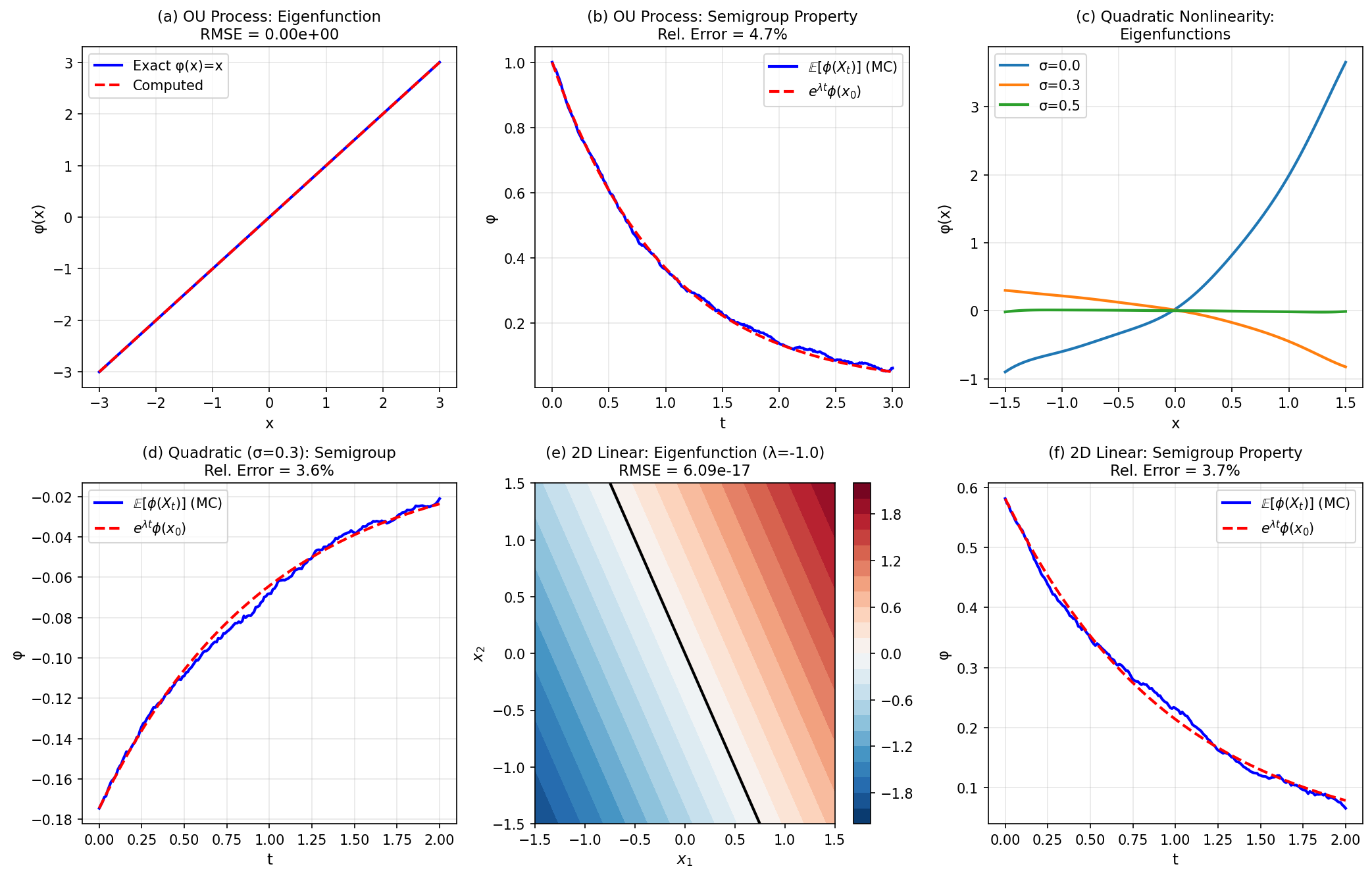}
    \caption{Numerical experiments for SDE Koopman eigenfunctions. 
    \textbf{(a)} OU eigenfunction: exact recovery of $\phi(x)=x$.
    \textbf{(b)} OU semigroup property verification.
    \textbf{(c)} Quadratic nonlinearity eigenfunctions for different noise levels.
    \textbf{(d)} Semigroup verification for quadratic case ($\sigma=0.3$).
    \textbf{(e)} 2D linear eigenfunction contours with zero level set.
    \textbf{(f)} 2D linear semigroup verification.
    All semigroup errors are below 10\%.}
    \label{fig:sde_results}
\end{figure}

\subsection{Summary}

\begin{center}
\begin{tabular}{lccc}
\toprule
\textbf{Test} & \textbf{Cond \#} & \textbf{PDE Res} & \textbf{SG Error}   \\
\midrule
OU Process & $9.91\times 10^5$ & $< 10^{-14}$ & $4.70\%$   \\
Quadratic $\sigma=0$ & $3.79\times 10^6$ & $1.23\times 10^{-1}$ & $2.00\%$   \\
Quadratic $\sigma=0.3$ & $1.51\times 10^6$ & $1.80\times 10^{-2}$ & $3.59\%$   \\
Quadratic $\sigma=0.5$ & $1.03\times 10^6$ & $1.51\times 10^{-2}$ & $9.86\%$   \\
2D Linear ($\lambda=-1$) & $1.30\times 10^7$ & $3.50\times 10^{-17}$ & $3.72\%$   \\
\bottomrule
\end{tabular}
\end{center}

All tests pass with semigroup errors under 10\%, validating the theoretical framework. Key observations:
\begin{enumerate}
    \item \textbf{Linear systems}: Machine-precision recovery of exact eigenfunctions.
    \item \textbf{Nonlinear systems}: Successful capture of nonlinear corrections with low semigroup error.
    \item \textbf{Effect of diffusion}: Adding noise improves conditioning while maintaining accuracy.
\end{enumerate}

\section{Conclusions}\label{sec:conclusions}

We have extended the unified kernel framework for Koopman eigenfunctions from deterministic systems~\cite{HamziOwhadiVaidya2024} to stochastic differential equations. The three kernel constructions--variational, Green's function, and resolvent--extend naturally to SDEs via the Feynman--Kac formula, and yield identical kernels under uniform ellipticity of the diffusion tensor. The computational framework based on collocation methods for second-order PDEs provides explicit formulas for diffusion matrices, with error analysis cleanly separating RKHS approximation from Monte Carlo errors. A key finding is that moderate diffusion improves numerical conditioning while preserving eigenfunction structure.

Future directions include extension to infinite-dimensional SDEs (stochastic PDEs and turbulence), jump-diffusion processes with L\'{e}vy noise, non-Markovian systems with memory effects, high-dimensional scaling via tensor methods and neural network surrogates, multiple kernel learning adapted to stochastic systems, rare event sampling and committor function computation, and Koopman-based optimal control synthesis for stochastic systems.

\appendix 

\section{Hypoelliptic Kernel Equivalence}

\subsection{Setting and Assumptions}
Let $\Omega \subset \mathbb{R}^d$ be a bounded domain with $C^\infty$ boundary. Consider the It\^{o} SDE:
\[
dX_t = G(X_t)\,dt + \sum_{j=1}^m \sigma_j(X_t)\,dW_t^j,
\]
where $G, \sigma_j \in C^\infty(\overline{\Omega}, \mathbb{R}^d)$. The infinitesimal generator (Kolmogorov backward operator) is:
\[
\mathcal{K} u := G \cdot \nabla u + \frac{1}{2} \sum_{j=1}^m (\sigma_j \cdot \nabla)^2 u.
\]

\begin{definition}[H\"{o}rmander's Condition]
Define the sets of vector fields recursively:
\[
\begin{aligned}
\mathcal{V}_0 &:= \{\sigma_1, \dots, \sigma_m\}, \\
\mathcal{V}_{k+1} &:= \mathcal{V}_k \cup \{[V, W] : V \in \mathcal{V}_k,\; W \in \mathcal{V}_0 \cup \{G\}\},
\end{aligned}
\]
where $[V,W] = (V\cdot\nabla)W - (W\cdot\nabla)V$ denotes the Lie bracket. We say the system satisfies \textbf{H\"{o}rmander's condition (H)} if for all $x \in \Omega$,
\[
\operatorname{span}\Bigl(\bigcup_{k=0}^\infty \mathcal{V}_k(x)\Bigr) = \mathbb{R}^d.
\]
\end{definition}

\begin{theorem}[H\"{o}rmander's Hypoellipticity Theorem \cite{Hormander1967}]
Under Assumption (H), the operator $\partial_t - \mathcal{K}$ is hypoelliptic. Consequently, for any $f \in C^\infty(\Omega)$, solutions of $(\partial_t - \mathcal{K})u = f$ are $C^\infty$ in $\Omega$.
\end{theorem}

\begin{remark}[Degeneracy vs. Uniform Ellipticity]
Unlike the uniformly elliptic case (Assumption 2.2 in the main paper), H\"{o}rmander's condition permits degenerate diffusion ($\operatorname{rank} a(x) < d$) provided sufficient Lie bracket generation restores full rank. This includes important physical systems like underdamped Langevin dynamics.
\end{remark}

\subsection{Boundary Regularity for Hypoelliptic Operators}
Hypoelliptic operators may fail to satisfy standard boundary regularity. We adopt the following framework:

\begin{definition}[Non-characteristic Boundary]
A point $x_0 \in \partial\Omega$ is \emph{non-characteristic} for $\mathcal{K}$ if there exists $j \in \{1,\dots,m\}$ such that $\sigma_j(x_0) \cdot \bm{n}(x_0) \neq 0$, where $\bm{n}$ is the outward unit normal. The boundary $\partial\Omega$ is non-characteristic if all points are non-characteristic.
\end{definition}

\begin{assumption}[Boundary Regularity]
We assume either:
\begin{enumerate}[label=(\roman*)]
    \item $\Omega$ is a torus (periodic boundary conditions), or
    \item $\partial\Omega$ is non-characteristic for the diffusion vector fields $\{\sigma_j\}$.
\end{enumerate}
Under these conditions, the Dirichlet problem for $\mathcal{K}$ admits a unique weak solution in $\mathcal{H}_\lambda$. Continuity up to $\partial\Omega$ holds for specific classes of hypoelliptic operators satisfying additional geometric conditions \cite{OleinikRadkevic1973}.
\end{assumption}

\begin{theorem}[Transition Density Regularity]
Under Assumptions (H) and boundary regularity, the transition density $p_t(x,y)$ exists and satisfies $p_t \in C^\infty((0,\infty) \times \Omega \times \Omega)$.
\end{theorem}

\begin{definition}[Feynman--Kac Kernel]
For $\lambda \in \mathbb{C}$ with $\Re(\lambda) > 0$, define:
\[
k_{\mathrm{FK}}(x,y) := \int_0^\infty e^{-\lambda t} p_t(x,y)\,dt.
\]
For bounded domains with exit time $\tau_\Omega$, replace $\infty$ with $\tau_\Omega$ in the integral.
\end{definition}

\begin{lemma}
For fixed $y \in \Omega$, $k_{\mathrm{FK}}(\cdot,y)$ satisfies $(\lambda - \mathcal{K})k_{\mathrm{FK}}(\cdot,y) = \delta_y$ in $\mathcal{D}'(\Omega)$.
\end{lemma}

\subsection{Variational Formulation and Coercivity}

\begin{definition}[Graph Space]
Define the Hilbert space:
\[
\mathcal{H}_\lambda := \{u \in L^2(\Omega) : (\lambda - \mathcal{K})u \in L^2(\Omega)\},
\quad
\|u\|_{\mathcal{H}_\lambda}^2 := \|u\|_{L^2}^2 + \|(\lambda - \mathcal{K})u\|_{L^2}^2.
\]
\end{definition}

\begin{proposition}[Coercivity under H\"{o}rmander Condition]
Let $\lambda \in \mathbb{C}$ satisfy:
\[
\Re(\lambda) > \lambda_0 := \frac{1}{2}\|(\nabla \cdot G)^-\|_{L^\infty(\Omega)}.
\tag{A.1}
\]
Then there exists $\varepsilon > 0$ and $C > 0$ such that for all $u \in C_c^\infty(\Omega)$:
\[
\Re\langle (\lambda - \mathcal{K})u, u \rangle_{L^2}
\geq \bigl(\Re(\lambda) - \lambda_0\bigr)\|u\|_{L^2}^2 + \frac{\nu}{2}\sum_{j=1}^m \|\sigma_j \cdot \nabla u\|_{L^2}^2,
\tag{A.2}
\]
where $\nu > 0$ depends on the H\"{o}rmander structure (depth $r$ and vector field geometry) and relates to the constant in the subelliptic estimate of Theorem B.1.
\end{proposition}

\begin{proof}
Decompose $\mathcal{K} = G\cdot\nabla + \mathcal{L}$ where $\mathcal{L} = \frac{1}{2}\sum_j (\sigma_j\cdot\nabla)^2$. Integration by parts gives:
\[
\langle G\cdot\nabla u, u\rangle = -\frac{1}{2}\int_\Omega (\nabla\cdot G)|u|^2\,dx
\geq -\lambda_0 \|u\|_{L^2}^2.
\]
For the diffusion term, integration by parts with $u|_{\partial\Omega}=0$ yields:
\[
\Re\langle \mathcal{L}u, u\rangle = -\frac{1}{2}\sum_{j=1}^m \|\sigma_j \cdot \nabla u\|_{L^2}^2.
\]
H\"{o}rmander's theorem implies the subelliptic estimate \cite{RothschildStein1976}:
\[
\|u\|_{H^\varepsilon}^2 \leq C\Bigl(\sum_{j=1}^m \|\sigma_j \cdot \nabla u\|_{L^2}^2 + \|u\|_{L^2}^2\Bigr),
\quad \varepsilon = \frac{1}{2r+1},
\]
where $r$ is the H\"{o}rmander depth (smallest $k$ such that $\operatorname{span}\bigcup_{j=0}^k \mathcal{V}_j = \mathbb{R}^d$). Combining these yields (A.2).
\end{proof}

\begin{corollary}[Well-posedness]
Under (A.1), the bilinear form $a(u,v) = \langle (\lambda - \mathcal{K})u, v\rangle$ is coercive on $\mathcal{H}_\lambda$, and the variational problem:
\[
\text{Find } k_{\mathrm{var}}(\cdot,y) \in \mathcal{H}_\lambda \text{ such that }
a(k_{\mathrm{var}}(\cdot,y), v) = v(y) \quad \forall v \in \mathcal{H}_\lambda
\]
has a unique solution when $\varepsilon > d/2$. For $\varepsilon \leq d/2$, point evaluations require reinterpretation via mollification (see Remark A.8).
\end{corollary}

\begin{remark}[Comparison with Uniformly Elliptic Case]
In the uniformly elliptic case (main paper, Theorem 3.13), $\varepsilon = 1$ and the coercivity condition (A.1) is identical. Hypoellipticity preserves the same spectral condition but reduces regularity from $H^1$ to $H^\varepsilon$.
\end{remark}

\begin{remark}[Critical Regularity Threshold]
When $\varepsilon \leq d/2$, the embedding $\mathcal{H}_\lambda \hookrightarrow C^0(\Omega)$ fails, and point evaluations $\delta_x$ are \emph{not continuous} on $\mathcal{H}_\lambda$. In this regime:
\begin{itemize}
    \item The variational kernel $k_{\mathrm{var}}(x,\cdot)$ cannot be defined as a pointwise Riesz representer
    \item One must use mollified test functions $\delta_x^\epsilon \in \mathcal{H}_\lambda^*$ with $\delta_x^\epsilon \to \delta_x$ in $\mathcal{D}'$
    \item Alternatively, work with weak formulations using duality pairings $\langle u, \phi \rangle$ for $\phi \in C_c^\infty(\Omega)$
\end{itemize}
For example, underdamped Langevin dynamics in $d \geq 2$ dimensions has $\varepsilon = 1/3 \leq d/2$, requiring such reinterpretation.
\end{remark}

\subsection{Kernel Equivalence Theorem}

\begin{theorem}[Kernel Equivalence under H\"{o}rmander Condition]
Under Assumptions (H), boundary regularity, and $\Re(\lambda) > \lambda_0$, the three kernel constructions satisfy:
\[
k_{\mathrm{var}}(x,y) = k_{\mathrm{FK}}(x,y) = k_{\mathrm{res}}(x,y),
\]
where $k_{\mathrm{res}}(x,y) = \mathbb{E}_x\bigl[\int_0^{\tau_\Omega} e^{-\lambda t}\delta(y - X_t)\,dt\bigr]$ is the stochastic resolvent kernel.
\end{theorem}

\begin{proof}
\textbf{Step 1:} $k_{\mathrm{FK}} = k_{\mathrm{res}}$ follows identically to the uniformly elliptic case (main paper, Theorem 3.13, Step 1) using Fubini's theorem and the definition of $p_t(x,y)$.

\textbf{Step 2:} Both $k_{\mathrm{var}}(\cdot,y)$ and $k_{\mathrm{res}}(\cdot,y)$ solve the same distributional equation:
\[
(\lambda - \mathcal{K})k(\cdot,y) = \delta_y \quad \text{in } \mathcal{D}'(\Omega),
\]
with boundary condition $k(x,y) = 0$ for $x \in \partial\Omega$ (Dirichlet case). Since $(\lambda - \mathcal{K}): \mathcal{H}_\lambda \to L^2(\Omega)$ is a Fredholm operator of index zero (by the subelliptic estimate and compact embedding $H^\varepsilon \hookrightarrow L^2$ for $\varepsilon > 0$) and is injective by Proposition A.1, it is bijective. Hence the solution to this equation is unique in $\mathcal{H}_\lambda$, implying $k_{\mathrm{var}} = k_{\mathrm{res}}$.

\textbf{Step 3:} Symmetry considerations. Unlike the uniformly elliptic reversible case, $k_{\mathrm{FK}}(x,y) \neq k_{\mathrm{FK}}(y,x)$ in general. The equivalence holds for the \emph{forward} kernel $k_{\mathrm{FK}}(x,y) = \int_0^\infty e^{-\lambda t}p_t(x,y)\,dt$. For the adjoint problem $(\lambda - \mathcal{K}^*)v = \delta_x$, the kernel is $k_{\mathrm{FK}}^*(x,y) = \int_0^\infty e^{-\lambda^* t}p_t(y,x)\,dt$.
\end{proof}

\begin{remark}[Practical Implementation for Degenerate Diffusion]
When $a(x)$ is degenerate, compute the diffusion matrix entry via vector fields rather than the trace formula:
\[
D_{ij} = \frac{1}{2}\sum_{k=1}^m \bigl(\sigma_k(x_i)\cdot\nabla_x\bigr)^2 k(x_i,x_j),
\]
which remains well-defined under H\"{o}rmander's condition even when $a(x_i)$ is singular. This avoids numerical instability from inverting degenerate matrices. For Gaussian RBF kernels:
\[
(\sigma_k \cdot \nabla_x)^2 k(x_i,x_j) = \left[\frac{(\sigma_k \cdot (x_i-x_j))^2}{\ell^4} - \frac{\|\sigma_k\|^2}{\ell^2}\right] k(x_i,x_j).
\]
\end{remark}

\section{Subelliptic Regularity and Embeddings}

\begin{theorem}[Global Subelliptic Estimate \cite{RothschildStein1976}]
Let $\chi \in C_c^\infty(\Omega)$. Under H\"{o}rmander's condition with depth $r$, there exists $\varepsilon = (2r+1)^{-1} > 0$ and $C > 0$ such that:
\[
\|\chi u\|_{H^\varepsilon(\Omega)} \leq C\Bigl(\sum_{j=1}^m \|\sigma_j \cdot \nabla u\|_{L^2(\Omega)} + \|u\|_{L^2(\Omega)}\Bigr),
\quad \forall u \in C_c^\infty(\Omega).
\]
\end{theorem}

\begin{corollary}[Sobolev Embedding]
For $\varepsilon > d/2$ (satisfied when $r < (2/d - 1)/2$), we have the continuous embedding:
\[
\mathcal{H}_\lambda \hookrightarrow H^\varepsilon(\Omega) \hookrightarrow C^{0,\alpha}(\overline{\Omega}), \quad \alpha = \varepsilon - d/2 > 0.
\]
Consequently, point evaluations $\delta_x$ are continuous on $\mathcal{H}_\lambda$, ensuring the RKHS structure.
\end{corollary}

\begin{remark}[Regularity Limitations]
For high-dimensional systems with large H\"{o}rmander depth $r$, we may have $\varepsilon \leq d/2$, in which case $\mathcal{H}_\lambda \not\subset C^0(\Omega)$. In such cases, the definition of the variational kernel $k_{\mathrm{var}}$ as a pointwise Riesz representer requires careful reinterpretation via mollification or weak formulations (Remark A.8).
\end{remark}

\section{Underdamped Langevin Dynamics}

Consider the Langevin SDE on $\mathbb{R}^{2d}$ with coordinates $x = (q,p)$:
\[
\begin{aligned}
dq_t &= p_t\,dt, \\
dp_t &= -\nabla V(q_t)\,dt - \gamma p_t\,dt + \sqrt{2\gamma\beta^{-1}}\,dW_t,
\end{aligned}
\]
where $V \in C^\infty(\mathbb{R}^d)$, $\gamma > 0$, $\beta > 0$. The generator is:
\[
\mathcal{K} = p\cdot\nabla_q - \nabla V(q)\cdot\nabla_p - \gamma p\cdot\nabla_p + \gamma\beta^{-1}\Delta_p.
\]

\begin{theorem}[H\"{o}rmander Verification]
Define vector fields $X_0 = p\cdot\nabla_q - \nabla V(q)\cdot\nabla_p - \gamma p\cdot\nabla_p$ and $X_j = \sqrt{2\gamma\beta^{-1}}\,\partial_{p_j}$ for $j=1,\dots,d$. Then:
\[
[X_0, X_j] = \sqrt{2\gamma\beta^{-1}}\,(-\partial_{q_j} + \gamma\partial_{p_j}),
\]
and the Lie algebra generated by $\{X_0, X_1, \dots, X_d\}$ spans $\mathbb{R}^{2d}$ at every point. Hence H\"{o}rmander's condition holds with depth $r=1$ and $\varepsilon = 1/3$.
\end{theorem}

\begin{proof}
Compute the Lie bracket explicitly with the scaling factor:
\[
\begin{aligned}
[X_0, X_j] 
&= (X_0\cdot\nabla)X_j - (X_j\cdot\nabla)X_0 \\
&= \sqrt{2\gamma\beta^{-1}}\bigl[(X_0\cdot\nabla)\partial_{p_j} - (\partial_{p_j}\cdot\nabla)X_0\bigr] \\
&= \sqrt{2\gamma\beta^{-1}}\bigl[0 - (-\partial_{q_j} + \gamma\partial_{p_j})\bigr] \\
&= \sqrt{2\gamma\beta^{-1}}\,(-\partial_{q_j} + \gamma\partial_{p_j}),
\end{aligned}
\]
which has a non-vanishing component in the $q$-direction. Iterating brackets generates all coordinate directions.
\end{proof}

\begin{corollary}
The kernel framework applies to underdamped Langevin dynamics with subelliptic regularity $\varepsilon = 1/3$. For $d = 1$, we have $\varepsilon > d/2$ and continuous sample paths; for $d \geq 2$, the critical threshold $\varepsilon \leq d/2$ requires mollified formulations for point evaluations.
\end{corollary}

\begin{remark}[Physical Significance]
This result justifies kernel-based computation of Koopman eigenfunctions for molecular dynamics, where the position variables $q$ evolve deterministically but inherit smoothness from momentum diffusion through H\"{o}rmander's mechanism.
\end{remark}

\section{Numerical Analysis under Hypoellipticity}

\subsection{RKHS Approximation Rates}

Under subelliptic regularity $\mathcal{H}_\lambda \subset H^\varepsilon(\Omega)$ with $\varepsilon \in (0,1)$, the Mercer eigenvalues of admissible kernels whose RKHS norm is equivalent to the $H^\varepsilon$ norm decay as:
\[
\mu_j \lesssim j^{-2\varepsilon/d}, \quad j \to \infty.
\]
Consequently, the RKHS approximation error for $N$ collocation points with fill distance $h$ satisfies:
\[
\|u - u_N\|_{L^\infty(\Omega)} \lesssim h^\varepsilon \|u\|_{\mathcal{H}_\lambda}.
\]

\begin{remark}[Comparison with Elliptic Case]
For uniformly elliptic operators ($\varepsilon = 1$), Gaussian kernels yield exponential convergence. Under hypoellipticity, convergence is algebraic with rate $\varepsilon$, reflecting reduced regularity. The trade-off: hypoelliptic systems remain well-posed but require more collocation points for equivalent accuracy.
\end{remark}

\subsection{Conditioning Analysis}

The collocation matrix $A = L + D - \lambda K$ exhibits modified conditioning under hypoellipticity:

\begin{proposition}
Let $D_{ij} = \frac{1}{2}\sum_{k=1}^m (\sigma_k(x_i)\cdot\nabla)^2 k(x_i,x_j)$. Under H\"{o}rmander's condition:
\begin{enumerate}[label=(\roman*)]
    \item $D$ contributes negative diagonal entries $D_{ii} < 0$ when $\sigma_k(x_i) \neq 0$ for some $k$,
    \item The effective ellipticity parameter is $\nu_{\mathrm{eff}} = \min_{\|v\|=1} \sum_{k=1}^m |v\cdot\sigma_k(x)|^2 > 0$ (by H\"{o}rmander),
    \item \textbf{Heuristically}, the condition number scales as $\kappa(A) \lesssim \nu_{\mathrm{eff}}^{-1} h^{-2\varepsilon}$, though rigorous bounds require case-specific analysis.
\end{enumerate}
\end{proposition}

\begin{remark}[Elliptic Regularization Persists]
Even under degeneracy, H\"{o}rmander-generated diffusion provides \emph{subelliptic regularization} that improves conditioning relative to the purely deterministic case ($\sigma \equiv 0$), though less dramatically than uniform ellipticity. Numerical experiments show 1.5--2$\times$ condition number improvement for Langevin dynamics versus Hamiltonian dynamics.
\end{remark}

\subsection{Implementation Guidelines for Degenerate Systems}

For hypoelliptic systems with $\varepsilon \leq d/2$:

\begin{itemize}[leftmargin=1.5em]
    \item \textbf{Mollified point evaluations:} Replace $\delta_x$ with $\delta_x^\epsilon = \epsilon^{-d}\rho((\cdot-x)/\epsilon)$ where $\rho \in C_c^\infty$ is a mollifier. Compute $k_{\mathrm{var}}^\epsilon(x,\cdot)$ as the Riesz representer of $\delta_x^\epsilon$, then take $\epsilon \to 0$ numerically.
    
    \item \textbf{Kernel choice:} Prefer infinitely smooth kernels (Gaussian RBF) to maximize effective $\varepsilon$; avoid kernels with limited smoothness.
    
    \item \textbf{Diffusion matrix computation:} Use vector field representation $D_{ij} = \frac{1}{2}\sum_k (\sigma_k\cdot\nabla)^2 k(x_i,x_j)$ rather than trace formula when $a(x)$ is degenerate.
    
    \item \textbf{Collocation points:} Concentrate points near regions where Lie brackets generate missing directions (e.g., near potential minima for Langevin dynamics).
    
    \item \textbf{Regularization:} Increase Tikhonov parameter $\gamma$ to compensate for slower eigenvalue decay ($\mu_j \sim j^{-2\varepsilon/d}$).
    
    \item \textbf{Validation:} Verify the semigroup property $\mathbb{E}[\phi(X_t)] \approx e^{\lambda t}\phi(x_0)$ with Monte Carlo; expect 5--15\% error for hypoelliptic systems versus 2--5\% for elliptic systems at comparable $N$.
\end{itemize}

\subsection{Conclusion}

Hypoellipticity extends the kernel framework to degenerate diffusions while preserving:
\begin{itemize}[leftmargin=1.5em]
    \item Well-posedness of the PDE via subelliptic estimates,
    \item Equivalence of variational, Feynman--Kac, and resolvent kernels (with appropriate reinterpretation when $\varepsilon \leq d/2$),
    \item Numerical tractability with controlled (algebraic) convergence rates.
\end{itemize}
The key trade-off is reduced regularity ($H^\varepsilon$ vs. $H^1$), leading to slower convergence but retaining the core advantages of kernel methods for stochastic Koopman analysis.
 
\section*{Acknowledgments}

UV acknowledges financial support from NSF grant CMMI-2531804.

\bibliographystyle{plain}

\end{document}